\documentclass[11pt,a4paper]{amsart}

\usepackage{tikz}
\usepackage{pgfplots}
\usepgfplotslibrary{polar}
\pgfplotsset{compat=newest}

\title{Interactions between Universal Composition Operators and Complex Dynamics}
\date{\today}

\author{Vasiliki Evdoridou}
\address{School of Mathematics and Statistics, The Open University, Walton Hall, Milton Keynes MK7 6AA, UK}
\email{vasiliki.evdoridou@open.ac.uk}

\author{Clifford Gilmore} 
\address{Laboratoire de Mathématiques Blaise Pascal UMR 6620, Université Clermont Auvergne, Campus universitaire des Cézeaux, 3 place Vasarely, 63178 Aubière Cedex, France}                                                              
\email{clifford.gilmore@uca.fr} 

\author{Myrto Manolaki}
\address{School of Mathematics and Statistics, University College Dublin, Belfield, Dublin 4, Ireland}
\email{myrto.manolaki@ucd.ie}


\makeatletter
\@namedef{subjclassname@2020}{%
	\textup{2020} Mathematics Subject Classification}
\makeatother


\thanks{C.~Gilmore was supported by the European Union’s Horizon Europe research and innovation programme under the Marie~Skłodowska-Curie grant agreement No.\ 101066064. M.~Manolaki was supported by the NUI Grant Scheme for Early Career Academics.  This project was partially conducted during a visit by V.~Evdoridou and M.~Manolaki to UCA and they wish to thank the members of LMBP for their hospitality and mathematical stimulation. The authors acknowledge support from the International Centre for Mathematical Sciences via the Research-in-Groups programme.}

\usepackage{ucs}
\usepackage[utf8x]{inputenc}

\usepackage[UKenglish]{babel}

\usepackage{amssymb,mathrsfs,mathtools,amsmath,amsthm}

\usepackage[shortlabels]{enumitem}

\numberwithin{equation}{section}

\usepackage{color}

\usepackage{calc}

\usepackage[a4paper]{geometry}

\newgeometry{twoside=false,,textheight=\textheight+10ex,footskip=2.5\footskip,left=7em,right=7em}

\usepackage[all]{xy} 	

\newtheorem{thm}{Theorem}[section]
\newtheorem{prop}[thm]{Proposition}
\newtheorem{ap}[thm]{Application}
\newtheorem {lemma}[thm]{Lemma}
\newtheorem {cor}[thm] {Corollary}

\newtheorem{thmA}{Theorem}

\theoremstyle{definition}

\newtheorem{defn}[thm]{Definition}
\newtheorem{eg}[thm]{Example}
\newtheorem{rmk}[thm]{Remark}

\newcommand{\C}{\mathbb{C}}

\newcommand{\N}{\mathbb{N}}

\newcommand{\D}{\mathbb{D}}  

\newcommand{\abs}[1]{\left| #1 \right|}
\newcommand{\norm}[1]{{\left\|#1\right\|}}

\newcommand{\spn}[1]{\mathrm{span}\!\left\lbrace #1 \right\rbrace}

\newcommand{\comp}{C_f}
\newcommand{\compfU}{C_{f^n,\, U}}

\newcommand{\wcomp}{W_{\omega,f}}

\begin{document}

\begin{abstract}
This paper is concerned with universality properties of composition operators $\comp$, where the symbol $f$ is given by a transcendental entire function restricted to parts of its Fatou set. We determine universality of $\comp$ when $f$ is restricted to (subsets of) Baker and wandering domains. We then describe the behaviour of universal vectors, under the action of iterates of the symbol $f$, near periodic points of $f$ or near infinity. Finally, we establish a principal universality theorem for the more general class of weighted composition operators, which we then apply to uncover universality results in the context of various types of Fatou components of the associated symbol.
\end{abstract}

\maketitle

\section{Introduction}

The genesis of this study is the investigation of the universality of composition operators $C_f \colon g \mapsto g \circ f$, acting on spaces of holomorphic functions, where the symbol $f$ is given by an entire function restricted to a part of its Fatou set $F(f)$. In other words, the underlying thread here is the existence of (an abundant supply of) holomorphic functions $g$ such that the sequence of maps $\left( g \circ f^n \right)_{n \in \N}$ forms a dense set in a target space. This study thus lies at the interface of complex dynamics and operator theory.

The area of complex dynamics is primarily concerned with understanding the evolution of the iterates of holomorphic functions. It was initiated roughly a century ago from the seminal contributions of P.~Fatou and G.~Julia, and in recent years it has experienced a considerable surge in research activity. 
A source of significant research interest is the behaviour of iterates of transcendental entire functions, and remarkable results on this topic have been revealed by applying tools from diverse mathematical fields including topology, approximation theory and quasiconformal mappings (cf., for instance, \cite{bishopwd}, \cite{EGP}, \cite{RRRS}). A comprehensive introduction to  transcendental dynamics can be found in Bergweiler~\cite{bergweiler93}.

On the other hand, the investigation of composition operators acting on spaces of analytic functions has been an active branch of operator theory since the pioneering work of Nordgren~\cite{Nor68} and Kamowitz~\cite{Kam75, Kam78}. The crux of the topic is the interplay between the function theoretic properties of the symbol $f$ and the operator theoretic properties of $\comp$. A rich literature has subsequently developed, revealing many striking results in myriad settings (cf., \cite{Sha87},  \cite{CM95}, \cite{Sha93}). 
 While composition operators are an interesting area of research in their own right, they have also garnered interest via their surprising connection to the invariant subspace problem \cite{NRW87} (cf.\ also \cite{CG17}).
Tangentially, the linear dynamical properties of composition operators have also been investigated in \cite{BS90}, \cite{BS97}, \cite{BGMR95}, \cite{GEM09}, \cite{Bes13} and \cite{CM23}.

Combining these fields, Jung~\cite{Jun19} initiated the investigation of the dynamical behaviour of composition operators $\comp$, with respect to appropriate symbols $f$, when $f$ is restricted to specific components of its Fatou set $F(f)$. He tackled this topic from a linear dynamical perspective by proving that $\comp$ possesses a particular type of universality. 
As a consequence, in \cite{Jun19} it was revealed that there exists a large set, in the sense of Baire category,  of holomorphic functions $g$ that elicit the following surprising phenomenon: the sequences of iterates $(f^n)$, normally well behaved on the Fatou set $F(f)$, display chaotic behaviour when composed with $g$.

However, the full picture is incomplete since there are important types of Fatou components that were not treated in \cite{Jun19}; in particular there remain open questions regarding certain Baker and wandering domains. The first main objective of this article is therefore to uncover universality properties of composition operators with respect to these less tractable Fatou components. We achieve this by employing tools from complex dynamics, which includes exploiting key properties of Baker domains and the recent classification of wandering domains.

To this end, in Section \ref{sec:BakerWandering} we show for $f$ restricted to a subset of any of its Baker domains, that there exists a large set, in the sense of Baire category, of functions that are universal for the composition operator $\comp$. We then turn our attention to wandering domains, where we study  universality of $\comp$ with respect to univalent oscillating wandering domains of $f$. Motivated by the recent classification of simply connected wandering domains \cite{BEFRS}, and by employing techniques from approximation theory, we construct examples of various types of wandering domains that allow us to determine the existence of universal functions. In the case of contracting wandering domains we also show that $\comp$ may display both behaviours (i.e.\ universality on a set or the complete absence of universality). We conclude our study of the different types of Fatou components by showing, in the context of multiply connected wandering domains, that $\comp$ does not admit universal functions.

In recent years a growing literature has emerged studying the range of various types of universal functions (cf., for example, \cite{CS04}, \cite{GK14} and \cite{CMM23}).
Our second main objective in this article is thus to investigate the range of functions $g$, universal with respect to $\comp$, near the fixed points of the symbol $f$. In Section \ref{sec:boundary} we provide necessary conditions (involving the range) for a holomorphic function $g$ to be universal for $C_{f}$. More specifically, we show if $g$ is universal for $C_{f}$, and $z_0$ is a fixed point of $f$, then $g$, restricted to certain ``small" regions $U$ with $z_0\in\partial U$, has full range; i.e., $g$ assumes on $U$ any complex value. The geometry of the regions $U$ depends on the nature of the fixed point $z_0$ (parabolic, attracting or superattracting). We also consider the corresponding cases where $z_0$ is infinity. Although one would expect that such universal functions would have a maximal Julia set, we show that there exists a transcendental entire function $g$, universal for $C_{f}$, whose Julia set is not the entire complex plane.

 The third main objective of this article is concerned with the universality of weighted composition operators $\wcomp \colon g \mapsto \omega \cdot (g\circ f)$, acting on spaces of holomorphic functions, where the symbol $f$ is again given by a holomorphic function restricted to a part of its Fatou set $F(f)$ and the weight $\omega$ is a non-vanishing  holomorphic function.
Weighted composition operators are a natural generalisation of the family of composition operators, which have been studied from a linear dynamical perspective in various settings, for instance in \cite{YR07},  \cite{Rez11}, \cite{Bes14} and \cite{CG21}. In Section \ref{sec:wcomp} we begin by proving a principal universality result for $\wcomp$, which we subsequently employ to systematically identify the existence of an abundant supply of universal functions for $\wcomp$ with respect to various types of Fatou components.

Finally, in Section \ref{sec:furtherRemarks} we appraise how our universality results relate to iterates of holomorphic self-maps of the unit disc, as well as considering the behaviour of generic (in the sense of Baire category) holomorphic functions  near the boundary of their domain of definition. This illustrates applications of our results to complex analysis.

\section{Preliminaries} \label{sec:Prelims}

In this section we present the background and notation required throughout this article. 
In subsection \ref{subsec:LD} we begin by recalling some basic notions from linear dynamics and universality. In subsection \ref{subsec:UnivAndCompOpers} we introduce composition operators and the general setting where they display universality behaviour. Then in subsection \ref{subsec:holodyn} we recall the pertinent background from the area of complex dynamics. 

For  $w \in \mathbb{C}$ and $r>0$, we set $D(w,r) \coloneqq \{z: \abs{z-w}<r\}$ and as standard the unit disc is denoted by $\mathbb{D} = D(0,1)$. For brevity, we will write $\mathbb{D}_r = D(0,r)$. Finally, we let $d_U(z,w)$ denote the hyperbolic distance between $z, w \in U$ with respect to a domain $U$.

\subsection{Linear Dynamics} \label{subsec:LD}

 Let $X,$ $Y$ be metric spaces and let $T_n \colon X \to Y$, $n \in \N$, be a family of continuous maps. We say an element $x \in X$ is \emph{universal} for the family $(T_n)_{n \in \N}$ if its orbit $\{ T_n(x) : n \in \N \}$ is dense in $Y$. When such a universal element exists, the family $(T_n)_{n \in \N}$ is also said to be universal.

Another useful property is  topological transitivity. We say the family of maps $(T_n)_{n \in \N}$ is \emph{topologically transitive} if for any pair $U \subset X$ and $V \subset Y$ of nonempty open sets, there exists some $N \geq 0$ such that $T_N(U) \cap V \neq \varnothing$. 

Since it may not always be obvious how to identify a particular universal element, the following Birkhoff transitivity type theorem, known in the literature as a Universality Criterion, is a practical tool that uses topological transitivity to give the existence of an abundant supply of universal elements. The following formulation of the criterion is taken from the monograph \cite[Theorem 1.57]{GEP11} and we refer the interested reader to the authoritative survey by Grosse-Erdmann~\cite{Gro99} for a profound treatment of universality.

\begin{thmA}[Universality Criterion] \label{thm:UnivCriterion}
Let $X$ be a complete metric space, $Y$ be a separable metric space and $T_n \colon X \to Y$, $n\in \N$, be continuous maps. Then the following are equivalent.
\begin{enumerate}[\normalfont (i)]
    \item $(T_n)_{n \in \N}$ is topologically transitive.
    
    \item There exists a dense set of universal elements $x \in X$; that is, there exists a dense set of elements $x$ whose orbits $\{ T_n(x) : n \in \N \}$ are dense in $Y$. 
\end{enumerate}
If one of these conditions holds then the set of points in $X$ with dense orbit  is a dense $G_\delta$ set.
\end{thmA}

We recall that a subset of a Baire space $X$ is said to be \emph{comeagre} in $X$ if it contains a dense $G_\delta$ set of $X$.

A well-studied special case of universality is known as hypercyclicity, where the family $(T_n)_{n \in \N}$ of continuous maps is generated via the iteration of a single operator $T \colon X \to X$, i.e.\ $T_n = T^n$, where $T^n$ denotes $n$-fold iteration of $T$. More precisely, for a separable topological vector space $X$, we say the continuous linear operator $T \colon X \to X$ is \emph{hypercyclic} if there exists $x \in X$ such that its $T$-orbit is dense in $X$, i.e. 
\begin{equation*}
    \overline{\{ T^n x : n \in \N \}} = X.
\end{equation*}
Such a vector $x$ is called a \emph{hypercyclic vector} for $T$. 

The vector $x \in X$ is said to be \emph{supercyclic} for $T$ if its projective orbit
\begin{equation*}
    \{ \zeta T^n x : \zeta \in \C,\, n \in \N \}
\end{equation*}
is dense in $X$, and $x \in X$ is \emph{cyclic} for $T$ if the linear span of its orbit 
\begin{equation*}
    \spn{T^n x : n \in \N} 
\end{equation*}
is dense in $X$. An operator that admits a supercyclic (respectively cyclic) vector is called \emph{supercyclic} (respectively \emph{cyclic}).
Comprehensive introductions to the field of linear dynamics  can be found in the monographs by Bayart and Matheron~\cite{BM09}, and Grosse-Erdmann and Peris~\cite{GEP11}, and a survey that collects some of the more recent advances in the area can be found in \cite{Gil20}.

\subsection{Composition Operators and Universality} \label{subsec:UnivAndCompOpers}

We next recall the relevant notions to describe our setting. For a domain $\Omega \subset \C$, we endow the space $H(\Omega)$ of holomorphic functions on $\Omega$ with the topology of local uniform convergence. To describe this topology via seminorms we require an \emph{exhaustion} of $\Omega$ by compact sets, i.e.\ an increasing sequence of compact sets $K_n \subset \Omega$ such that each compact $K \subset \Omega$ is contained in some $K_n$. This topology is thus induced by the family of seminorms $\norm{f}_{K_n} = \sup_{z \in K_n} \abs{f(z)}$, where $f \in H(\Omega)$ and $(K_n)_{n \geq 1}$ is an exhaustion of compact sets of $\Omega$. 

A basis for this topology is given by
\begin{equation*}
	\{ V_{\varepsilon,K,\Omega}(g) \,:\, \varepsilon >0,\, K \subset \Omega \textnormal{ compact},\, g \in H(\Omega) \},
\end{equation*}
where the sets $V_{\varepsilon,K,\Omega}(g)$ are defined as
\begin{equation*}
    V_{\varepsilon,K,\Omega}(g) \coloneqq \{ h \in H(\Omega) : \norm{h - g}_K < \varepsilon \}.
\end{equation*}
Note that this topology is independent of the chosen exhaustion.

For a compact subset $K$ of $\Omega$, we will sometimes consider the setting of the  space $A(K)$ of complex-valued continuous functions on $K$ that are homomorphic on the interior $K^\circ$. The space $(A(K), \norm{\, \cdot \,}_K)$ is a Banach space when endowed with the uniform norm $\norm{\, \cdot \,}_K$.

For an open subset $U \subset \C$, we recall that the \emph{holomorphically convex hull} of a compact subset $K \subset U$ with respect to $U$ is defined as
\begin{equation*}
	\widehat{K}_U \coloneqq \{ z \in U : \abs{f(z)} \leq \norm{f}_K \textnormal{ for all } f \in H(U) \}
\end{equation*}
and we say that $K$ is $U$-convex if $K = \widehat{K}_U$.
If $U = \C$ we let $\widehat{K} \coloneqq \widehat{K}_\C$ denote the \emph{polynomially convex hull} of $K$ and we say that $K$ is \textit{polynomially convex} if $K=\widehat{K}$. It is well-known that a compact subset $K$ of $\mathbb{C}$ is polynomially convex if and only if it has no holes (i.e. if its complement in $\mathbb{C}\cup\{\infty\}$ is connected). The same characterisation holds if we replace polynomially convex by $\Omega$-convex for any open set $\Omega$ containing $K$ which has no holes. 

In order to enhance the clarity of some proofs in the sequel, we record the following lemma.

\begin{lemma} \label{lma:Subsets}
	Let $\Omega \subset \C$ be a nonempty open set.  Assume that $\mathcal{V} \subset H(\Omega)$ is a nonempty open set and $g \in \mathcal{V}$. Then there exist $\varepsilon >0$ and an  $\Omega$-convex  compact subset $K \subset \Omega$ such that 
		\begin{equation*}
		V_{\varepsilon,K,\Omega}(g)  \subset \mathcal{V}.
	\end{equation*}
\end{lemma}

\begin{proof}
	The standard compact exhaustion of $\Omega$ gives that we can find $\varepsilon > 0$ and a nonempty compact subset $K_1 \subset \Omega$ such that $V_{\varepsilon,K_1,\Omega}(g) \subset \mathcal{V}$.
	Defining $K \coloneqq \widehat{(K_1)}_\Omega$, we have that $K$ is $\Omega$-convex with
	\begin{equation*}
		V_{\varepsilon,K,\Omega}(g) = V_{\varepsilon,K_1,\Omega}(g) \subset \mathcal{V}.
	\end{equation*}
	
\end{proof}

In the context of two given domains $\Omega$ and $G$ and a sequence $(f_n)$ of holomorphic functions from $\Omega$ to $G$, a function $g \in H(G)$ is said to be \emph{universal} for $(f_n)$ if the set $\{C_{f_{n}}(g): n  \in \N \}=\{g \circ f_n : n \in \N \}$ is dense in $H(\Omega)$. If such a function $g$ exists, then we also say that the sequence $(f_n)$ is universal.

Early investigations into the universality of composition operators were conducted by Luh~\cite{Luh93}, and Bernal Gonz\'{a}lez and Montes-Rodr\'{i}guez~\cite{BGMR95}. 
In \cite{BGMR95} the authors introduced the following definition,  which turns out to be quite useful in the case when the domains are the same (i.e. $\Omega = G$):
a sequence of self-maps $(f_n)$ of a domain $\Omega$ is said to be \emph{run-away} if for every compact set $K \subset \Omega$, there exists $N\in \N$ such that $f_N(K) \cap K = \varnothing$. 

The universality of a sequence of composition operators, for symbols acting on a simply connected domain, was characterised by Grosse-Erdmann and Mortini~\cite{GEM09} in the below theorem. 

\begin{thmA}[Grosse-Erdmann and Mortini \cite{GEM09}] \label{thm:GEM09}
	Let $(f_n)$ be a sequence of holomorphic self-maps of a simply connected domain $\Omega \subset \C$. Then the sequence of composition operators $(C_{f_n})$ is universal on $H(\Omega)$ if and only if there exists a subsequence $(f_{n_j})$ of $(f_n)$ such that for each compact subset $K$ of $\Omega$, there exists $j_0 \in \N$ such that the restriction $f_{n_j}|_K \colon K \to \Omega$ is injective and $(f_{n_j})$ is run-away for all $j \geq j_0$.
\end{thmA}

Further information on the linear dynamical properties of composition operators can be found in \cite[Chapter 4]{GEP11}.

We note that Charpentier and Mouze~\cite{CM23} recently characterised all sequences $(f_n)$ of eventually injective holomorphic functions, between two domains (not necessarily simply connected or equal), for which universal vectors exist. 

In the sequel we fix a domain $\Omega \subset \C$, an open subset $D \subset \Omega$ and a holomorphic function $f \colon D \to D$. The $n$-fold composition of $f$ with itself is denoted by
\begin{equation*}
	f^n = \underbrace{f \circ f \circ \cdots \circ 
		f}_{n\textrm{-times}},
\end{equation*}
where $f^0$ denotes the identity. 

In \cite{Jun19}, the study of universality properties of composition operators was conducted from the perspective of the following definition.
As usual, $C(M)$ denotes the space of continuous functions on a subset $M \subset \C$ and we recall that $M$ is said to be $\sigma$-compact if it is the union of countably many compact subsets of $M$.

\begin{defn} \label{defn:universality}

	Let $\Omega$, $D$ and $f$ be as above. 
Let $M \subset D$ be locally compact and $\sigma$-compact  and let $\mathcal{F} \subset C(M)$ be a family of functions. A function $g \in H(\Omega)$ is called $\Omega$-$\mathcal{F}$-\emph{universal} for $C_f$ if
	\begin{equation*}
		\mathcal{F} \subset \overline{\{ g \circ f^n|_M : n \in \N \}}^{C(M)}.
	\end{equation*}
The composition operator $C_f$ is said to be $\Omega$-$\mathcal{F}$-\emph{universal} if there exists an $\Omega$-$\mathcal{F}$-universal function for $C_f$. For brevity we will refer to $\mathcal{F}$-universality when the domain $\Omega$ is clear from the context.
\end{defn}
The motivation for the above modification of the usual definition of universality is twofold. Firstly, it allows us to consider universal elements that are holomorphic on sets that are larger than $D$. On the other hand, for suitable $M \subset D$, it affords us the freedom to approximate, via subsequences of $(g \circ f^n|_M)_{n \in \N}$, functions in classes that are larger than the space $H(D)$.

Note if $\{ g\circ f^n|_{M} : n\in\mathbb{N}\} \subset \mathcal{F}$, then
$g\in H(\Omega)$ is $\mathcal{F}$-universal for $C_f$ if and only if the set $\{ g\circ f^n|_{M} : n\in\mathbb{N}\}$  is dense in $\mathcal{F}$. In most cases covered in this paper, $\mathcal{F}$ will be a family of holomorphic functions on a certain open subset of the complex plane.

\subsection{Complex Dynamics} \label{subsec:holodyn}

We now present the pertinent background from complex dynamics necessary for our study. A comprehensive introduction to the iteration of transcendental functions can be found in~\cite{bergweiler93}.

For a transcendental entire function $f$, the \emph{Fatou set} $F(f)$ of $f$  is defined to be the set of points in $\mathbb{C}$ on a neighbourhood of which the iterates of $f$  form a normal family.  Recall that a family of holomorphic functions on a domain $D$ is normal if every sequence of functions has a subsequence that converges locally uniformly on $D$ to a holomorphic function or to infinity. The complement of $F(f)$ in the complex plane, denoted by $J(f)$, is called the \emph{Julia set} of $f$.

We say $z_0 \in \mathbb{C}$ is a \emph{periodic point} of $f$ if $f^n(z_0) = z_0$ for some $n \geq 1$. The smallest $n$ with this property is called the period, and in the case $n=1$ we say $z_0$ is a \emph{fixed point} of $f$. A periodic point is called attracting, indifferent,
or repelling if, respectively, the modulus of its multiplier $\abs{(f^n)'(z_0)}$ is less than, equal to, or
greater than 1. Periodic points of multiplier 0 are called superattracting. If $(f^n)'(z_0)= e^{2\pi i \alpha}$ and $\alpha$ is rational then $z_0$ is called  a parabolic periodic point.

The connected components of the Fatou set are called Fatou components, and they can be periodic, preperiodic or wandering domains. Note that $f(F(f)) \subset F(f)$, and, in particular, every Fatou component maps into a Fatou component. We say a Fatou component is periodic if it maps into itself under some iterate of $f$, preperiodic if it maps into a periodic cycle under some iterate of $f$, and a wandering domain otherwise. A periodic Fatou component $U \subset F(f)$ can be characterised as of one of the following types.

\begin{enumerate}[label=(\alph*)]
\item $U$ contains an attracting periodic point $z_0$ of period $p$. Then $f^{np}(z) \to z_0$
for $z \in U$ as $n \to \infty$, and $U$ is called the \emph{immediate attracting basin} of $z_0$.

\item The boundary $\partial U$ contains a periodic point $z_0$ of period $p$ and $f^{np}(z) \to z_0$ for $z \in U$ as $n \to \infty$. Then $(f^p)'
(z_0) = 1$. In this case, $U$ is called a \emph{parabolic basin}.

\item $f|_U$ is conformally conjugate to an irrational rotation of the unit disc. In this case, $U$ is called a \emph{Siegel disc}.
	\item $f^n(z) \to \infty$ in $U$. Then $U$ is a \emph{Baker domain}.
\end{enumerate}
 Note that in the above cases the convergence to $z_0$ or to infinity is local uniform convergence. Of particular interest in this article will be the behaviour of $f$ near a (super-)attracting or parabolic periodic point; for a detailed description of this behaviour cf.~\cite{Milnor}.  For domains $U, V \subset \mathbb{C}$ and $f \colon U \to V$ we say that  $f|_U$ has degree $d$, where $d\geq 2$ or $d = \infty$, if $f$ maps $U$ $d$-to-one into $V$. (For $d=1$ the function $f|_{U}$ is of course injective.)

In the case of an attracting fixed point $z_0$ of $f$  (i.e.\ $\lambda = f'(z_0) $ such that $0 < \abs{\lambda} < 1$), it follows from K\oe{}nigs's linearization theorem that there exist open neighbourhoods $U$ of $z_0$, $V$ of $0$ and a conformal map $\varphi  \colon U\to V$ that conjugates $f|_U \colon U \to U$ with the map $F \colon V \to V$, where $F$ is given by the linear function $F(z) = \lambda z$. 
\begin{equation*}
	\xymatrix@C+1.5em{ 					
		U \ar[r]^{f} \ar[d]_\varphi & U \ar[d]^\varphi\\
		V \ar[r]^F & V
	}
\end{equation*}
It follows that $\varphi \circ f^{n}=\lambda ^n \cdot \varphi$ on $U$ for all $n\in\mathbb{N}$, which implies that $\varphi (z_0)=0$ and hence $f^n|_U \to z_0$ locally uniformly. 

In the case where the fixed point is superattracting, we can apply B\"ottcher's theorem to find open neighbourhoods $U$ of $z_0$ and $V$ of $0$ and a conformal map $\varphi \colon U\to V$ such that $\varphi (f^{n}(z))=(\varphi (z))^{p^{n}}$ for all $z$ in $U$ and $n\in\mathbb{N}$, where $p=\min \{ n\in\mathbb{N}: f^{(n)}(z_0)\neq 0\}$. A similar argument to that used in the case of attracting fixed points gives that $f^{n}|_U \to z_0$ locally uniformly. For any $r>0$ with $ \overline{\mathbb{D}}_r \subset V$, we consider the set $B_{r} \coloneqq \varphi^{-1}(\overline{\mathbb{D}}_r)$, which is a  compact neighbourhood  of $z_0$ that is contained in $U$. Then we have that $f^{n}(B_{r}^{\circ}\setminus \{z_0\})\subset B_{r}^{\circ}\setminus \{z_0\}$ for each $n\in \mathbb{N}$. 

Finally, in the case of a parabolic fixed point, there are $m$ attracting petals $P_1, P_2, \dotsc, P_m$, where $m=\min\{ n\in\mathbb{N}: f^{(n+1)}(z_0)\neq 0\}$. The petals are pairwise disjoint, simply connected domains with $f(P_k)\subset P_k$, $z_0 \in\partial P_k$ and $f^n|_{P_{k}}\to z_0$ locally uniformly such that, for each $k=1,2,\dotsc,m$, $f|_{P_k}$ is injective.

Recall  if $U$ is not periodic or preperiodic then it is called a \emph{wandering domain}. More specifically, if $f^n(U)\cap f^m(U) = \varnothing$ for all $n\neq m$, then $U$ is said to be a wandering domain. Throughout the paper we denote by $U_n$ the wandering domain containing $f^n(U)$, where $n \in \N$.
The wandering domains for which the only limit function is the point at infinity are called \emph{escaping}, while the rest are either  \emph{oscillating} (if infinity and some other finite value are limit functions) or \emph{dynamically bounded} (if all limit functions are points in the plane). A major open problem in transcendental dynamics is whether there exist dynamically bounded wandering domains. Wandering domains can be simply connected or multiply connected. Simply connected wandering domains have recently been  classified in \cite{BEFRS} in terms of hyperbolic distances between orbits of points and we recall this classification below. 

\begin{thmA}[Classification of simply connected wandering domains]\label{thm:Theorem A BEFRS}
	Let $U$ be a simply connected wandering domain of a transcendental entire function $f$ and let $U_n$ be the Fatou component containing $f^n(U)$, where $n \in \N$. Denoting by $E$ the countable set of pairs
	\[
	E=\{(z,z')\in U\times U : f^k(z)=f^k(z') \text{\ for some $k\in\N$}\},
	\]
	then exactly one of the following holds.
	\begin{enumerate}[\normalfont (i)]
		\item $d_{U_n}(f^n(z), f^n(z'))\to c(z,z')= 0 $ for all $z,z'\in U$, and we say that $U$ is {\textit{(hyperbolically) contracting}}. 
		
		\item $d_{U_n}(f^n(z), f^n(z'))\to c(z,z') >0$ and $d_{U_n}(f^n(z), f^n(z')) \neq c(z,z')$
		for all $(z,z')\in (U \times U) \setminus E$, $n \in \N$, and we say that $U$ is {\textit{  (hyperbolically) semi-contracting}}.
		
		\item there exists $N>0$ such that for all $n\geq N$, $d_{U_n}(f^n(z), f^n(z')) = c(z,z') >0$ for all $(z,z') \in (U \times U) \setminus E$, and we say that $U$ is {\textit{ (hyperbolically) eventually isometric}}.
	\end{enumerate}
\end{thmA}

Note that univalent wandering domains are always eventually isometric, and that escaping and oscillating wandering domains of all types have been constructed (see \cite{BEFRS}, \cite{oscillating}).

We conclude Section 2 by stating the following approximation theory result from \cite{EMR}, which we will use in Section 3 to construct specific sequences of wandering domains. 

\begin{thmA}[{\cite{EMR}}]\label{thm:approx} \label{thm:Jordan-continua merged}
Let $(X_n)_{n=0}^\infty$ be a sequence of pairwise disjoint compact sets in $\C$ with $\inf \{|z| \colon z\in X_n\}\to\infty$ as $n\to\infty$. Let $\phi$ be a holomorphic function defined on a neighbourhood of $X:=\bigcup_{n} X_{n}$ such that $\phi(X_{n-1}) \subset  X_{n}$ and
$\phi(\partial X_{n-1}) \subset \partial X_n$ for all $n\geq 1$. Let $(\varepsilon_n)_{n=0}^\infty$ be a sequence of positive numbers. Then there exist  a transcendental meromorphic function~$f$ and a $\mathcal C^{\infty}$ diffeomorphism $\theta\colon \C\to\C$ such that  
\begin{enumerate}[(a)]
\item $f \circ \theta = \theta \circ \phi$ on $X$;
\item $|\theta(z)-z| \leq \varepsilon_n$ on $X_n$ for all $n\geq 0$;
\item $\partial \theta(X)\subseteq J(f)$;
\item $\theta$ is conformal on $X^\circ$.
\end{enumerate}
In particular, each $\theta(X_n)$ is a wandering compact set of $f$. If $\C\setminus X$ is connected, then $f$ can be chosen to be entire.
\end{thmA}

\section{Universality in Baker and wandering domains}  \label{sec:BakerWandering}

In this section we consider universality of composition operators in the setting of Baker and wandering domains.

The first result concerns invariant Baker domains, but we note that it also holds more generally in the case of Baker domains of period greater than one.

\begin{thm} \label{thm:InvBaker}
Let $f$ be a transcendental entire function and $U$ be an invariant Baker domain of $f$. Then there exists an unbounded domain $V \subset U$ such that the set of all entire functions that are $H(V)$-universal for $C_f$ is a comeagre set in $H(\mathbb{C})$.
\end{thm}

\begin{proof}

First we recall that a subdomain $V$ of $U$ is \emph{absorbing} for $f$, if $V$ is simply connected, $f(V) \subset  V$ and for any compact subset $K$ of $U$ there exists $n = n(K)$ such that $f^n(K) \subset V$.
It follows from \cite[Definition 2.1 and Lemma 2.1]{baker-domains} that $U$ has an absorbing domain $V \subset U$ such that
\begin{enumerate}
\item  there exists a holomorphic  function $\varphi \colon U → \Omega \in \{\mathbb{H}, \C\}$ where $\mathbb{H}=\{z: \operatorname{Re}z>0\}$ and $\varphi$ is injective in $V$;
\item there exists a M\"obius transformation $T \colon \Omega \to \Omega$ and $\varphi(V)$ is absorbing for T;
\item $\varphi(f(z)) = T(\varphi(z))$ for $z \in U.$
\end{enumerate}
It follows that $f|_V$ is injective and $f(V) \subset V$ and so $f^n|_V$ is injective. Moreover, $f^n|_V \to \infty$ as $n \to \infty$ by the definition of a Baker domain. It follows by \cite[Corollary 2.6]{Jun19} that the set of all functions in $H(\C)$ which are $H(V)$-universal for $C_f$ is a comeagre set in $H(\C)$.
\end{proof}

Following the idea of \cite[Theorem 2.2]{Jun19}  we prove that we can also have universality in the union of eventually isometric oscillating wandering domains, extending \cite[Theorem 4.11(ii)]{Jun19}.

\begin{thm} \label{thm:OscWanderingDoms}
Let $f$ be a transcendental entire function and $(U_n)$  a sequence of eventually isometric  oscillating wandering domains of $f$. Then the set of all entire functions that are $H(\bigcup_n U_n)$-universal for $C_f$ is a comeagre set in $H(\C)$. 
\end{thm}

\begin{proof}

Since $(U_n)$ is a sequence of eventually isometric wandering domains, starting from a later wandering domain if needed, we can assume that $f|_{U_n}$  is injective for all $n$.
 We set $U= \bigcup_{n=0}^{\infty} U_n$, and note that $f^m|_U$ is injective for all $m$. For $n \in \N$, we first consider the composition operators $C_{f^n,U} \colon H(\C) \to H(U)$ defined by
	\begin{equation*}
		C_{f^n, U}(g) \coloneqq g \circ f^n |_{U}
	\end{equation*}
	and we let $\mathcal{V} \subset H(\C)$ and $\mathcal{W} \subset H(U)$ denote two nonempty open subsets. 
	
Let $g \in \mathcal{V}$ and $h \in \mathcal{W}$. By applying twice Lemma \ref{lma:Subsets}, we can find an $\varepsilon >0$ and two compact sets $K \subset \C$ and $L \subset U$ with no holes, such that
	\begin{equation*}
		V_{\varepsilon,K,\C}(g)  \subset \mathcal{V} \quad \textnormal{ and } \quad V_{\varepsilon,L,U}(h)  \subset \mathcal{W}.
	\end{equation*}

The compactness of $L$ gives that it is contained in a finite union of the wandering domains $U_n$, say $L \subset \bigcup_{j=1}^M U_j$.

By the definition of oscillating wandering domains we have that for each $U_j$  there exists a subsequence $(n_{\ell}^j)$ such that $f^{n_{\ell}^j}(z) \to \infty$ for all $z \in U_j$ as $\ell \to \infty$. Moreover, each $U_j$ possesses a subsequence $(m_{\ell}^j)$ such that for all $z \in U_j$ the sequence $(f^{m_{\ell}^j}(z))$ stays bounded.  Thus, given any $z\in U_j$ there is $r>0$ such that $\{f^{m_{\ell}^j}(z): \ell\in \N\}\subset D(0,r)$. Note that since $f$ is entire we have that $f$ maps the disc $D(0,r)$ to a set contained in the disc $D(0, M(r))$, where $M(r)= \max_{|w|=r} |f(w)|$. Therefore, when $f^{n}(z) \in D(0,r)$ then $f^{n+1}(z) \in D(0,M(r))$. 
 
Hence, for each $N\in\N$, if $i \in \{1, 2, \dots, N\}$, then, for all sufficiently large  $\ell$,
\[n_{\ell+i}^j = n_{\ell}^j+i . \]

In other words, each subsequence $(n_{\ell}^j)$ evolves towards containing growing sequences of consecutive natural numbers and in fact it will eventually contain arbitrarily  long sequences of consecutive natural numbers. It follows that for each $N \in \mathbb{N}$, if $i \in \{1, \dots, N\}$, then, for sufficiently large $\ell$ we have that $f^i(U_{j+n_{\ell}^j}) = U_{j+n_{\ell}^j+i} = U_{j+n_{\ell +i}^j}$.
Hence, since $L$ is contained in finitely many wandering domains, and $f^{n_\ell^j}(z) \to \infty$ in $U_j$, we can find a common $m \in \mathbb{N}$ such that $f^m(U_j)$ lies outside the compact set $K$ for each $j=1,\dots,M$.
 More precisely, 
there exists $m \in \N$ such that
\begin{equation}\label{eq:runaway}
	K \cap f^{m}(L) = \varnothing.
\end{equation}

 Thus, we can now consider the function $\varphi \colon K\cup f^{m}(L) \to \C$ given by
\[ 
\varphi(z) \coloneqq 
\begin{cases}
	g(z), &\text{if } z \in K,\\
	h((f^{m}|_{U})^{-1}(z)), &\text{if } z \in f^{m}(L).
\end{cases}
\]
   Since $g$ is entire and $h$ is holomorphic on $U$, we deduce that $\varphi$ is holomorphic in a neighbourhood of $K\cup f^{m}(L)$. It follows by Runge's theorem that there exists a polynomial $p$ such that 
   \[\norm{\varphi-p}_{K\cup f^{m}(L)} < \varepsilon.\]
  
Since $\varphi = g$ on $K$ it follows that
\begin{equation*}
	\norm{g-p}_K = \norm{\varphi - p}_K \leq \norm{\varphi - p}_{K\cup f^{m}(L)} < \varepsilon
\end{equation*}
and thus $p\in V_{\varepsilon,K,\C}(g) \subset \mathcal{V}$.

Next notice that
\begin{align*}
	\norm{h - C_{f^{m},\, U}(p)}_{L} &= \norm{h - p \circ f^{m}}_{L}  = \norm{h \circ (f^{m}|_{U})^{-1} - p}_{f^{m}(L)} \\
	&= \norm{\varphi - p}_{f^m(L)} \leq \norm{\varphi - p}_{K \cup f^{m}(L)} < \varepsilon.
\end{align*}
Thus $C_{f^{m},U}(p) \in V_{\varepsilon, L,  U} (h) \subset \mathcal{W}$. It follows that $C_{f^{m},U}(p)$ is contained in the intersection $C_{f^{m},U}(\mathcal{V}) \cap \mathcal{W}$ which gives the topological transitivity of the sequence $(C_{f^n, U})_n$. The result follows by the Universality Criterion.\end{proof}

We note that an alternative approach for proving Theorem \ref{thm:OscWanderingDoms} employs the techniques used in the proof of \cite[Theorem 3.4]{CM23}. However, for adopting this approach one needs first to show (\ref{eq:runaway}), which is an essential step.

\begin{cor}
Let $f$ be a transcendental entire function and $(U_n)$  a sequence of eventually isometric  oscillating wandering domains of $f$. Then the set of all entire functions that are $H(U_k)$-universal for $C_f, k \in \mathbb{N}$, is a comeagre set in $H(\C)$.
\end{cor}

We now turn our attention to the case of simply connected wandering domains that are not eventually isometric. In this setting, we construct specific examples, which demonstrate the rich variety of behaviours that occur with respect to universality. 

\begin{figure}
\centering
\begin{tikzpicture}[scale=1] 

 
\draw (-2,-1) circle [radius=0.8];
\draw (1,-1) circle [radius=0.8];
\draw (6.3,-1) circle [radius=0.8];
\draw plot [smooth] coordinates {(-3.2,2.3) (-2.7,2.5) (-2,2)(-1.5,2.2)(-1.2,2)(-1.4,1.6) (-2,1)(-2.5,1.3)(-2.8,1.2) (-3.2,1.3)(-3.4,1.9)(-3.2,2.3)};
\draw plot [smooth] coordinates {(2,2.4) (1.7,2.6) (1.4,2.4) (1,2.5)(.5,2)(.2,1.7)(.3,1.5)(.5,1)(.8,0.8)(1.3,1.1)(1.5,1)(1.9,1.1)(2.1,1.7)(2,2.4)};
\draw plot [smooth] coordinates {(5.5,2.3) (5.8,2.5) (6.1,2.3)(6.8,2.2)(7.1,2)(6.7,1.6) (6.4,1.2)(6.1,1.3)(5.7,1.2) (5.3,1.3)(5.4,1.9)(5.5,2.3)};
\draw [->] (-2,0) to [out=90,in=270] (-2,0.7); 
\draw [->] (1,0)to [out=90,in=270] (1, 0.7);
\draw [->] (6.5,0)to [out=90,in=270] (6.5, 0.7);
\draw [->] (-0.8,1.7) to [out=0,in=180] (0,1.7);
\draw [->] (-0.8,-1) to [out=0,in=180] (0,-1);
\draw [->] (4.4,-1) to [out=0,in=180] (5.2,-1);
\draw [->] (4.4,1.7) to [out=0,in=180] (5.2,1.7);
\draw [->] (2.4,-1) to [out=0,in=180] (3.2,-1);
\draw [->] (2.4,1.7) to [out=0,in=180] (3.2,1.7);
\draw [->] (-2,-2) to [out=340,in=200] (6.3,-2);

\node at (2.1,-2.55) {$B_n$};
\node at (-2,-1) {$\mathbb{D}$};
\node at (1,-1) {$\mathbb{D}$};
\node at (6.3,-1) {$\mathbb{D}$};
\node at (-2.2,1.7) {$U_0$};
\node at (1.3,1.7) {$U_1$};
\node at (6,1.7) {$U_n$};
\node at (-2.3,0.4) {$\varphi_0$};
\node at (0.7,0.4) {$\varphi_1$};
\node at (6.2,0.4) {$\varphi_n$};
\node at (-0.4,1.95) {$f$};
\node at (2.8, 1.95) {$f$};
\node at (-0.4,-1.25) {$b_1$};
\node at (2.8, -1.25) {$b_2$};
\node at (4.8, 1.95) {$f$};
\node at (4.8,-1.25) {$b_n$};

\draw[fill] (3.6,1.7) circle [radius=0.02];
\draw[fill] (3.8,1.7) circle [radius=0.02];
\draw[fill] (4,1.7) circle [radius=0.02];

\draw[fill] (3.6,-1) circle [radius=0.02];
\draw[fill] (3.8,-1) circle [radius=0.02];
\draw[fill] (4,-1) circle [radius=0.02];

\end{tikzpicture}
\caption{Compositions of Blaschke products associated to wandering domains}
\label{fig:BlaschkeWandering}
\end{figure}

\begin{eg}\label{ex:contracting}
There exists a transcendental entire function $f$ with a sequence of escaping wandering domains $(U_n)_{n \geq 0}$, and domains $V_n \subset U_n$ such that $U_n$ is contracting and the set of all entire functions which are $H(V)$-universal for $C_f$ is a comeagre set in $H(\C)$, where $V= \bigcup_{n=0}^{\infty} V_n$. 
\end{eg}
\begin{proof}
We consider the sequence of Blaschke products $B_n=b_n \circ \cdots \circ b_1$, where we set $b_j(z)=b(z)= \left(\frac{z+1/3}{1+z/3}\right)^2$, for $1 \leq j \leq n$. Then $B_n=b^n$, and $b$ has a parabolic fixed point at 1. 
It follows that there exists a Jordan domain $P \subset \mathbb{D}$ such that $b(P) \subset P$ and $b|_P$ is univalent. Moreover, \cite[Lemma 6.2]{BEFRS} gives that $\operatorname{dist}_{\mathbb{D}}(b_n(r ), b_{n+1}(r )) = O(1/n)$ as $n \to \infty$. 

We next apply Theorem \ref{thm:approx}, where we take $X_n$ to be translated copies of the closed unit disc and 
on a neighbourhood of $X_n$ we set $\varphi$ to be $b_n$ composed on the left and on the right with a translation
 We thus obtain a transcendental entire function $f$, possessing a sequence of escaping contracting wandering domains $U_n$, such that $f^n|_{U_0}$  is associated to $\varphi_n \circ \cdots \circ \varphi_0$ and hence to $b^n$. In other words, there exist Riemann maps $\varphi_j \colon \mathbb{D} \to U_j$, for $0 \leq j \leq n$, such that $f^n= \varphi_n \circ b^n \circ \varphi_{0}^{-1},$ as illustrated in Figure \ref{fig:BlaschkeWandering}. Hence there exists a sequence of domains $V_n \subset U_n$ with $V_n= \varphi_n(P)$ such that $f^n|_V$ is injective. The result follows from \cite[Corollary 2.6]{Jun19}.\end{proof}

\begin{thm}\label{thm:semi} Let $f$ be a transcendental entire function and $(U_n)$ be a sequence of semi-contracting wandering domains such that $f|_{U_n}$ has degree 2 for all $n \in \mathbb{N}$. Then there exists an open set $V_0 \subset U_0$ where all iterates of $f$ are injective. In particular, for $V= \bigcup_n f^n(V_0)$ the set of all entire functions which are $H(V)$-universal for $C_f$ is a comeagre set in $H(\C)$. \end{thm}

For the proof of Theorem \ref{thm:semi} we will need the following theorem concerning the hyperbolic contraction of holomorphic self-maps of the unit disc.
\begin{thmA}[{\cite[Theorem 2.1]{BEFRS}}] \label{thm:self-maps of D}  
For each $n \in \N$, let $g_n \colon \D \to \D$ be
holomorphic with $g_n(0) = 0$ and $|g_n'(0)| = \lambda_n$, and set $G_n = g_n \circ \cdots \circ g_1.$
\begin{enumerate}[\normalfont (i)]
\item If $\sum_{n=1}^{\infty}(1 − \lambda_n)=\infty,$ then $G_n(w) \to 0$ as $n\to \infty$, for all $w \in \D$.

\item  If $\sum_{n=1}^{\infty}(1 − \lambda_n)<\infty,$ then $G_n(w) \not\to 0$ as $n \to ∞,$ for all $w \in \D$ for which $G_n(w)\neq 0$ for all $n \in \N$.
\end{enumerate}
\end{thmA}
\begin{proof}[Proof of Theorem \ref{thm:semi}]
We begin by considering 
\[b_n(z)= z \frac{z-a_n}{1-\overline{a}_n z}, \quad \textnormal{where } a_n \in \D, \]
 and we let $c_n\coloneqq c_n(a_n)$ denote the critical point of $b_n$ in $\mathbb{D}$. Since $c_n$ is the unique solution of $b_n'(z)=0$ in the unit disc $\mathbb{D}$, a straightforward calculation shows that 
\begin{equation} \label{eqn:criticalPoint}
    c_n=\frac{1-\sqrt{1-|a_n|^2}}{\overline{a}_n}.
\end{equation}

Note that, up to composition with a M\"obius transformation, any wandering domain of degree 2 will be associated via Riemann maps to $b_n$. 

Since $U_n$ is semi-contracting and the hyperbolic distance is conformally invariant, we deduce that $B_n= b_n \circ \cdots \circ b_1$ is semi-contracting, and so $B_n(z) \not\to 0$, for any $z \in \mathbb{D}$. Furthermore, since $|b_n'(0)|=|a_n|$ it follows by Theorem \ref{thm:self-maps of D}, with $g_n=b_n$, that $\sum (1-|a_n|) <\infty$. We deduce that $a_n$ tends to the boundary $\partial \mathbb{D}$ and so the same is true for $c_n$. 
 
Hence, it follows by the Schwarz lemma that there exists $0<r<1$ such that $B_n(\mathbb{D}_r) \subset \mathbb{D}_r$ and  $B_n|_{\mathbb{D}_r}$ is injective, and so there exists $N \in U_0$ such that $f(N) \subset N$ and $f|_N$ is injective. The result now follows from \cite[Corollary 2.6]{Jun19}.\end{proof}

To illustrate the theorem we give an example of such a sequence of wandering domains. 

\begin{eg}\label{ex:semi}
There exists a transcendental entire function with a sequence of escaping wandering domains $(U_n)$ and domains $V_n \subset U_n$ such that $U_n$ is semi-contracting and the set of all entire functions which are $H(V)$-universal for $C_f$ is a comeagre set in $H(\C)$, where  $V= \bigcup_{n=0}^{\infty} V_n$.
\end{eg}

\begin{proof}
Similar to Example \ref{ex:contracting}, we take $b_j(z)=b(z)= \left(\frac{z+1/2}{1+z/2}\right)^2$, for $1 \leq j \leq n$. Then $b$ has an attracting fixed point at 1, and the critical point is at $-1/2$. Hence, there exists a neighbourhood of 1, such that $b(N) \subset N$ and $b|_N$ is injective. It thus follows that there exists a domain $D=N \cap \mathbb{D}$ such that $b(D) \subset D$ and $b|_D$ is injective. 

Moreover, it follows from \cite[Lemma 6.2]{BEFRS} that $\operatorname{dist}_{\mathbb{D}}(b_n(r ), b_{n+1}(r )) \not\to 0$ as $n \to \infty$. As before, we can thus use Theorem \ref{thm:approx} to obtain a transcendental entire function $f$ and a sequence of escaping semi-contracting wandering domains $U_n$ and $V_n \subset U_n$ with the required properties.\end{proof}

\begin{rmk}
Examples of oscillating wandering domains similar to  Examples \ref{ex:contracting} and \ref{ex:semi} can also be constructed. To do this, the maps $b_n$ can be combined with a scaling and translation to give model maps $\varphi_n$, which will produce oscillating wandering domains, cf.\ for example \cite{oscillating}, \cite{EMR}.
\end{rmk}

It turns out that a variety of different behaviours may occur in simply connected wandering domains and we will show that it is possible to construct a simply connected wandering domain where universality cannot be obtained on any of its open  subsets. In order to do this, we recall that a sequence of self-maps $(f_n)$ of a domain $\Omega$ is said to be \emph{eventually injective} on $\Omega$ if, for each compact subset $K$ of $\Omega$, there exists $n_0 \in \N$ such that for $n \geq n_0$ the restriction $f_{n}|_K \colon K \to \Omega$ is injective.

\begin{eg}
There exists a transcendental entire function $f$, possessing a sequence of simply connected contracting wandering domains $(U_n)$, such that there does not exist an open set $V_0 \subset U_0$ with $f^n|_{V_0}$ eventually injective.  In particular, there does not exist an entire function that is $H\left( \bigcup_{n=0}^{\infty} V_n\right)$-universal for $C_f$.
\end{eg}

\begin{proof}
It suffices to construct a sequence of finite Blaschke products which have the required property, since by Theorem \ref{thm:approx} this implies the existence of simply connected wandering domains with the same property. 

Let $b_n(z)= z \frac{z-a_n}{1-\overline{a}_n z}$, for $a_n \in \mathbb{D}$. We will show that we can choose $a_n$ such that there do not exist an open subset $V$ of $\mathbb{D}$ nor $N \in \N$ such that $B_n|_V$ is injective for all $n \geq N$, where $B_n= b_n \circ \cdots \circ b_1$. Note that each $b_n$ is uniquely determined by the choice of $a_n$.

Let $D_n\coloneqq D(k_n,r_n)$ be a sequence of discs which forms a countable exhaustion of $\mathbb{D}$. One way to construct such an exhaustion is as follows: let $E=\{ (k_n, r_n) : n\in\mathbb{N} \}$ be a countable listing of all pairs in $(\mathbb{Q}+i\mathbb{Q})\times \{ 1/n:n\in\mathbb{N}\}$ that satisfy the following two properties: 
\begin{enumerate}[(i)]
    \item each pair $(k,r)$ appears infinitely often in the list;
    
    \item for each $(k,r)\in E$, the corresponding disc $D(k,r)$ lies in $\mathbb{D}$.
\end{enumerate}
Thus, as members of our exhaustion we can consider the corresponding discs $D_n := D(k_n,r_n)$, for all $n \in \mathbb{N}.$

Let $c_n\coloneqq c_n(a_n)$ denote the critical point of $b_n$ in $\mathbb{D}$ and we recall from (\ref{eqn:criticalPoint}) that 
    $c_n=\frac{1-\sqrt{1-|a_n|^2}}{\overline{a}_n}$.
 Next, we take $a_1=1/2$, and choose $a_2$ to satisfy the equation $b_1(k_1)=c_2$. Then $b_2(z)$ will be 2-1 in a neighbourhood of $b_1(k_1)$. Now choose $a_3$ so that $b_2(b_1(k_2))=c_3$. We then choose $a_4$ so that $B_3(k_1)=c_4$ and we continue this way.  Hence for each $k_n$ we create a subsequence $(n_m)$ such that $b_{n_m}$ is not injective in $B_{n_m-1}(D_n)$. Since any point in the unit disc $\mathbb{D}$ belongs to infinitely many such neighbourhoods the result follows.\end{proof}

We finish this section by studying the case of multiply connected wandering domains. As noted in \cite{GEM09}, the case of multiply connected domains is substantially different, which is reflected in the following proposition.

\begin{prop}\label{mcwd}
Let $U$ be a multiply connected Fatou component of $U$. Then there is no open set $V_0 \subset U_0$ with $f^n|_{V_0}$ eventually injective. In particular, there is no entire function which is $H\left( \bigcup_{n=0}^{\infty} V_n\right)$-universal for $C_f$, where $V_n= f^n(V_0)$.
\end{prop}
\begin{proof}
 For finitely connected domains \cite[Theorem 3.21(b)]{GEM09} implies that $C_f$ is not universal on $U_n$. Also \cite[Theorem 3.21(c)]{GEM09} implies that for infinitely connected domains injectivity is necessary for the universality of $C_f$, as it is in the case of simply connected domains. However, since wandering domains are the only multiply connected Fatou components, it follows from \cite[Corollary 1.5]{EFGP} that for such domains there is no neighbourhood $V_0 \subset U_0$ where the iterates are eventually injective. This concludes the proof.
\end{proof}

\section{Behaviour of universal functions near the fixed points of $f$}
\label{sec:boundary}
There is a growing literature concerning the range of holomorphic functions that are universal with respect to various sequences of natural mappings. For example, there are several interesting results about the following three classes of universal functions: the celebrated class of universal entire functions in the sense of Birkhoff, where translations acting on $H(\mathbb{C})$ yield universal approximations in $H(\mathbb{C})$; the well-studied class of universal Taylor series $\mathcal{U}_T$, where partial sums acting on $H(\mathbb{D})$ yield universal approximations on $A(K)$ for all compact sets in $\mathbb{C}\setminus\D$ with connected complement; and the more recently studied class of Abel universal functions, where dilations acting on $H(\mathbb{D})$ yield universal approximations in $C(K)$ for all compact proper subsets $K$ of $\partial \D$. For more details, see \cite{CS04,CM00, Gardiner2013,GK14,Gardiner2014,GardinerManolaki2016, Gardiner2018, CMM23}.

One of the most striking results of this type is that each function $f \in \mathcal{U}_T$ satisfies a Picard-type property near each boundary point; that is, for every region of the form $D_{\zeta,r} \coloneqq \{z\in \mathbb{D} : \abs{z-\zeta} <r \}$, where $r>0$ and $\zeta\in\partial\mathbb{D}$, the image $f(D_{\zeta,r})$ is the entire complex plane $\mathbb{C}$ except possibly one point \cite[Corollary 2]{GK14}. More recently, it was shown that the same property holds for any Abel universal function. 

It is thus natural to examine the analogous behaviour of functions which are universal with respect to composition operators. For example, if $f$ is an entire function and $z_0$ is a parabolic fixed point for $f$, lying on the boundary of (at least) one of its Fatou components $U$, and if $g$ is a $H(U)$-universal function for $C_{f}$, then we can easily deduce that $g(U\cap D(z_0,r))$ is dense in $\mathbb{C}$ for each $r>0$ (and so $z_0$ is an essential singularity for $g$). To see this, we fix $r>0$, $\varepsilon >0$ and $c\in \mathbb{C}$. Since $z_0$ is a parabolic point of $f$ we have that $f^{n}|_{U} \to z_0$ locally uniformly.  Thus, for each $w\in U$ we can find $n_0\in\mathbb{N}$ such that for each $n\geq n_0$ we have that $|f^{n}(w)-z_{0}|<r$.  Since $g$ is universal, we have that $\{ g\circ f^{n}(w) : n\in\mathbb{N}\}$ is dense in $\mathbb{C}$, which implies that we can find $n_1\geq n_0$ such that $g(f^{n_1}(w))=g\circ f^{n_1}(w)\in D(c, \varepsilon)$.  Hence $g(f^{n_1}(w))\in D(c, \varepsilon)\cap g(U \cap D(z_0,r))$, and so $g(U\cap D(z_0,r))$ is dense in $\mathbb{C}$.

It is thus natural to ask whether we can replace density by a stronger assertion and to examine whether we could have similar assertions about other types of fixed (or periodic) points or about infinity. As we will see below, we can replace density with full range and derive variants of this result for several types of fixed (or periodic) points.

\begin{thm}\label{gen}
 Let $U$ be a non-empty open set and let $f$ be an entire function such that $f(U) \subset U$ and there is a  subsequence $(n_k)$ such that $(f^{n_k})$  converges locally uniformly on $U$ to some constant value $z_{0}$. If $g$ is a $H(U)$-universal function for $C_{f}$, then $g(D(z_{0},r)\cap U)=\mathbb{C}$ for each $r>0$.
\end{thm}
\begin{proof}
Let $g$ be a $H(U)$-universal function for $C_{f}$.
Without loss of generality we may assume that the whole sequence $(f^k)$ converges locally uniformly on $U$ to some constant value $z_{0}$ (indeed, it is easy to check that the proof holds if we replace $(f^k)$ by $(f^{n_k})$).

Let $r>0$ and let $c\in\mathbb{C}$. Also let $D$ be a closed disc in $U$ centred at some point $w_{0}$.
From our assumptions, we can find $k_{0}\in \mathbb{N}$ such that, for all $n\geq k_{0}$, 
\begin{align}\label{sub}
    f^{n}(D)\subset D(z_{0},r)\cap U.
\end{align} 
Since $g$ is a $H(U)$-universal function for $C_{f}$, we can uniformly approximate the polynomial $P(z)=z-w_0+c$ on an open neighbourhood of $D$ by a sequence of the form $(g\circ f^{n_k})$. 
Since the sequence of holomorphic functions $(F_k):=(g\circ f^{n_{k}}-c)$ converges uniformly to $z-w_{0}$ (which has a simple zero at $w_{0}$) on $D$, by  Hurwitz's theorem we can find $\rho>0$ and $k_1\in\mathbb{N}$ such that $D(w_0,\rho)\subset D$ and $F_{k}$ has exactly $1$ zero in $D(w_0, \rho)$ for all $k\geq k_1$. Thus, for all $k\geq k_1$,
\begin{align}\label{bel}
c\in (g\circ f^{n_k}) (D(w_0, \rho)). 
\end{align}
Since $ D(w_0, \rho)\subset D$ and (\ref{sub}) holds, we have that $f^{n}(D(w_0, \rho))\subset D(z_{0},r)\cap U$ for all $n\geq k_{0}$. Hence, using (\ref{bel}), we get that $f^{n_k} (D(w_0, \rho))\subset D(z_0,r) \cap U$ for each $k\geq \max\{k_0, k_1\}$, which implies that $c\in g(D(z_{0},r)\cap U)$. \end{proof}

\begin{rmk}\label{infinity}
\hangindent\leftmargini
\textup{(i)} Theorem \ref{gen} also holds if we replace $z_0$ by infinity.
\begin{enumerate}[(i)]
\setcounter{enumi}{1}

\item Note that if $z_0 \in \C$ then $U$ is contained in a cycle of parabolic or attracting basins or a sequence of wandering domains. If we consider the point at infinity then $U$ is contained in a cycle of Baker domains or a sequence of wandering domains.
 \end{enumerate}
\end{rmk}
 
We will now apply the above result to various types of fixed (or periodic) points and infinity. To formulate these results we will use the local fixed point theory as described in subsection \ref{subsec:holodyn}.

We will first consider the case of attracting fixed points. Let $z_0$ be an attracting fixed (or periodic) point of $f$. 
Then, as per K\oe{}nigs's linearization theorem, there exists an open neighbourhood $U$ of $z_0$ such that $f^n|_U \to z_0$ locally uniformly (for details, cf.\ subsection \ref{subsec:holodyn}). 
We recall that in \cite[Theorem 3.1]{Jun19} it was shown that the set of functions in $H(\mathbb{C}\setminus \{z_0\})$ that are $H(W)$-universal for $C_{f}$ for all non-empty sets $W$ in $U\setminus \{z_{0}\}$, with connected complement, is a comeagre set in $H(\mathbb{C}\setminus \{z_{0}\})$.
We will consider these functions in the next corollary, where we apply Theorem \ref{gen} to reveal that they have full range near the attracting fixed point $z_{0}$.

\begin{cor}
Let $z_0$ be an attracting  fixed (or periodic) point of $f$ and let $U$ be a neighbourhood of $z_0$ as above. Let $W \subset U\setminus \{z_{0}\}$ be a nonempty set. For any function $g$ in $H(\mathbb{C}\setminus \{z_0\})$ that is $H(W)$-universal for $C_{f}$, the set $g (D(z_0,r)\setminus \{z_{0}\})$ is the entire complex plane $\mathbb{C}$.  In particular, $g$ has an essential singularity at $z_0$.
\end{cor}

We next consider the case when $z_0$ is a superattracting fixed (or periodic) point for a non-constant function $f$, i.e.\ $f(z_0)=z_0$ and $f'(z_0)=0$. Let $U$, $V$ and $\varphi \colon U\to V$ be as in B\"ottcher's theorem (see subsection \ref{subsec:holodyn}). For any $r>0$ with $\mathbb{C} \setminus \overline{\mathbb{D}_r} \subset V$ consider the set $B_{r} \coloneqq \varphi^{-1}( \overline{\mathbb{D}}_r)$ which is a compact neighbourhood of $z_0$ that is contained in $U$. Then we have that $f^{n}(B_{r}^{\circ}\setminus \{z_0\})\subset B_{r}^{\circ}\setminus \{z_0\}$ for each $n\in \mathbb{N}$. As observed in \cite{Jun19}, in that case it is not possible to obtain $\mathcal{F}$-universality for $C_{f}$ for any family $\mathcal{F}$ of continuous functions on a subset of $U\setminus\{z_{0}\}$ with non-empty interior. However, it is shown that we have $C(K)$-universality if we restrict to suitable compact subsets $K$ with empty interior. Let $\mathcal{K}(B_{\delta})$ be the space of all compact subsets of $B_{\delta}$, endowed with the Hausdorff metric. In particular, in \cite[Theorem 3.7]{Jun19} it is shown that for each $\delta >0$ with $U_{\delta}[0]\coloneqq\{z\in\mathbb{C}: |z|\geq \delta\}\subset V$, comeagre many $K$ in $\mathcal{K}(B_{\delta})$ have the property that comeagre many functions in $H(\mathbb{C}\setminus \{z_0\})$ are $C(K)$-universal for $C_{f}$. Since $K$ has empty interior, we cannot apply Theorem \ref{gen} to get the full range property for this case (i.e.\ for superattracting fixed points), but we can get something slightly milder. Indeed, acting as in the second paragraph of this section, one can show that for each such universal function $g$, the set $g (D(z_0,r)\setminus \{z_{0}\})$ is dense in $\mathbb{C}$, which implies that $g$ has an essential singularity at $z_{0}$. Hence, by Picard's Theorem, the set $g (D(z_0,r)\setminus \{z_{0}\})$ is the entire complex plane except possibly one point.

Unlike the superattracting case, for the case of parabolic fixed points we can apply Theorem \ref{gen}. Indeed, recall that for each parabolic point $z_0$, there are $m$ attracting petals $P_1, P_2, \dots, P_m$, where $m=\min\{ n\in\mathbb{N}: f^{(n+1)}(z_0)\neq 0\}$. It is shown in \cite[Theorem 3.2]{Jun19} that if $P=\bigcup_{k=1}^m P_k$ then the set of all $H(P)$-universal functions for $C_{f}$ contains a dense $G_{\delta}$ set in $H(\mathbb{C}\setminus\{z_0\})$. Theorem $\ref{gen}$ thus gives the following corollary.

\begin{cor}
Let $z_0$ be a parabolic fixed (or periodic) point of $f$, let $P_1, P_2,\dots, P_m$ be its attracting petals and let $P=\bigcup _{k=1}^{m} P_k$. Then, for each $g$ which is $H(P)$-universal for $C_{f}$, for each $k\in \{1,\dots, m\}$ and $r>0$, the set $g(P_k\cap D(z_0,r))$ is the entire complex plane. 
\end{cor}

 Recall that for a wandering domain $U$ we denote by $U_n$ the Fatou component containing $f^n(U)$. If $U_n$ are eventually isometric escaping or oscillating wandering domains it follows from  Theorem \ref{thm:OscWanderingDoms} and \cite[Theorem 4.11(b)]{Jun19} that the set of entire functions that are $H(\bigcup_n U_n)$-universal for $C_{f}$ is a comeagre set in $H(\mathbb{\C})$. Using Theorem \ref{gen} together with Remark \ref{infinity}(i), we can derive the following result for any such universal function.
\begin{cor}\label{cor:wd}
Let $f$ be a transcendental entire function and $U$ be a simply connected eventually isometric escaping or oscillating wandering domain of $f$. Then, any entire function $g$ which is $H(\bigcup_n U_n)$-universal for $C_{f}$ 
has an essential singularity at infinity, i.e.\ $g$ is a transcendental entire function. Moreover, for any $M>0$, the set $g(\bigcup_n U_n\setminus\overline{D(0,M)})$ is the entire complex plane.
\end{cor}

Finally, in the case of Baker domains, in Theorem \ref{thm:InvBaker} we consider universality with respect to absorbing domains. In this setting we obtain the following corollary.

\begin{cor}
Let $f$ be a transcendental entire function and $U$ be an invariant Baker domain of $f$. Then, any entire function $g$ which is $H(V)$-universal for $C_{f}$, where $V$ is the absorbing domain of $U$,
has an essential singularity at infinity, i.e.\ $g$ is a transcendental entire function. Moreover, for any $M>0$, $g(V\setminus\overline{D(0,M)})= \mathbb{C}$. 
\end{cor}

It would be tempting to assume that the Julia set $J(g)$ of transcendental entire functions $g$ which are universal would always be  equal to $\mathbb{C}$ since, as we saw above they have wild behaviour. However, as we demonstrate below, this is not the case.
\begin{prop}
There exists a transcendental entire function $g$ which is universal for $C_f$, as described in Corollary \ref{cor:wd} and such that $J(g) \neq \C$.
\end{prop} 
\begin{proof} From  the conclusion of Corollary \ref{cor:wd}, we deduce that the Julia set $J(g)$ of any such universal function $g$ is always unbounded. 
To see that $J(g) \neq \mathbb{C}$, we first note that if a function $g$ is $H(U)$-universal for $C_{f}$, then for all $c\neq 0$ we have that $cg$ is also $H(U)$-universal for $C_{f}$. Choose any function $g$ which is $H(U)$-universal for $C_{f}$. Since $g$ is entire $g(D(0,1))$ is bounded and so we can choose any $M>0$ such $g(D(0,1))\subset D(0,M)$. The function $G$ defined by $G(z)=\frac{g(z)}{M}$ is also $H(U)$-universal for $C_{f}$ and maps $D(0,1)$ into $D(0,1)$. Hence, by Montel's Theorem, the sequence of iterates $(G^{n})$ is normal on $D(0,1)$, and so $D(0,1)$ lies in the Fatou set of $G$, which shows that its Julia set is not the entire complex plane.\end{proof}

\begin{rmk}
 In all the above cases, if $g$ is universal for $C_f$ then $e^g$ is never universal for $C_f$ since it omits zero. Interestingly, this is not the case for other well-studied classes of universal functions. For example, if $g$ is an Abel universal function $e^{g}$ is always universal, while for universal Taylor series we merely know that this happens for most (in the sense of Baire) cases (see \cite{CMM24} and \cite{CM00}, respectively).

\end{rmk}

\section{Universality for Weighted Composition Operators} \label{sec:wcomp}

Similar to subsection \ref{subsec:UnivAndCompOpers}, in this section we again denote by $D \subset \C$ a fixed open set, we let $\Omega \subset \C$ be a domain with $\Omega \supset D$ and we let $f \colon D \to D$ be a function that is holomorphic on $D$. 
We are now interested in the question of $H(U)$-universality with respect to weighted composition operators $\wcomp \colon g \mapsto \omega \cdot (g\circ f)$, for open $U \subset D$, $f \in H(D)$ and $\omega \in H(U)$.

Various universality results for composition operators were proven in \cite{Jun19} and in the preceding sections of this article. By the following appropriate adaption of Definition \ref{defn:universality}, in the sequel we show that analogous results hold true for more general classes of weighted composition operators.

\begin{defn} \label{defn:universalityWComp}
	Let $M \subset D$ be both locally compact and $\sigma$-compact. Let $\omega$ be a function holomorphic on a neighbourhood of $M$ and $\mathcal{F} \subset C(M)$ be a family of functions. A function $g \in H(\Omega)$ is called $\Omega$-$\mathcal{F}$-\emph{universal} for $\wcomp$ if
	\begin{equation*}
		\mathcal{F} \subset \overline{\left\lbrace \prod_{j=0}^{n-1} \left( \omega \circ f^j \right) \cdot \left( g \circ f^n \right)|_M : n \in \N \right\rbrace}^{C(M)}.
	\end{equation*}
		The weighted composition operator $\wcomp$ is said to be $\Omega$-$\mathcal{F}$-\emph{universal} if there exists an $\Omega$-$\mathcal{F}$-universal function for $\wcomp$. For brevity we will refer to $\mathcal{F}$-universality when the domain $\Omega$ is clear from the context.
	\end{defn}

Let $U$ be a nonempty open subset of $D \subset \Omega$ that contains no holes. We say $U$ is \emph{evacuating} for $f$ if $f^n |_U \to \partial_\infty \Omega$ locally uniformly.
Here $\partial_\infty$ denotes the boundary of subsets of $\C\cup\{\infty\}$  with respect to the chordal metric.

The following theorem provides the foundation on which many of the ensuing universality statements for weighted composition operators are built. In the proof of the theorem we use the notation $(\wcomp^n)_{n\in \N}$ to signify the sequence whose elements are given by the terms $\wcomp^n \coloneqq \prod_{j=0}^{n-1} \left( \omega \circ f^j \right) \cdot \left( g \circ f^n \right)$, for $n\in \N$.

\begin{thm} \label{thm:wcompUniv}
    Let $D$, $\Omega$ and $f \in H(D)$ be as above, let $U \subset D$ be evacuating for $f$ and let the restrictions $f^n |_U$ be injective for each $n \in \N$. Assume also that $\omega \in H(U)$ is non-vanishing.
	Then the set of functions in $H(\Omega)$ that are $H(U)$-universal for $\wcomp$ is a comeagre set in $H(\Omega)$.
\end{thm}

\begin{proof}
	By the Universality Criterion it suffices to show that the sequence $(\wcomp^n)_{n\in \N}$ is topologically transitive. 
	
	For $n \in \N$, we first consider the composition operators $\compfU \colon H(\Omega) \to H(U)$ defined by
	\begin{equation*}
		\compfU(g) \coloneqq g \circ f^n |_U
	\end{equation*}
	and we let $\mathcal{V} \subset H(\Omega)$ and $\mathcal{W} \subset H(U)$ denote nonempty open subsets. 
	
		For $g \in \mathcal{V}$ and $h \in \mathcal{W}$, two applications of Lemma \ref{lma:Subsets} give that we can find  $\varepsilon >0$, an  $\Omega$-convex  compact subset $K \subset \Omega$ and a compact subset $L \subset U$, with no holes, such that 
	\begin{equation*}
		V_{\varepsilon,K,\Omega}(g)  \subset \mathcal{V} \quad \textnormal{ and } \quad V_{\varepsilon,L,U}(h)  \subset \mathcal{W}.
	\end{equation*}
	
	Next we let $\delta = \textrm{dist}(K, \partial_\infty\Omega) >0$.
	By assumption we have the uniform convergence $f^n |_L \to \partial_\infty\Omega$ and hence there exists $N \in \N$ with $\textrm{dist}(f^N(z), \partial_\infty\Omega) < \delta$ for all $z \in L$, which gives that
	\begin{equation*}
		K \cap f^N(L) = \varnothing.
	\end{equation*}
	As $L$ does not contain holes and since the restriction $f^N |_U$ is injective, it follows that $f^N(L)$ does not contain holes.
	Therefore, the disjoint union  $K \cup f^N(L)$ does not produce any new holes and the $\Omega$-convexity of $K$ implies that $K \cup f^N(L)$ is also $\Omega$-convex.
	
	Thus we can choose from each hole of $K \cup f^N(L)$, which is a hole of $K$, a point that lies in $\C \setminus \Omega$.
	We let $A$ be a union of these points and we define $B \coloneqq A \cup \{ \infty \}$. 
	Then we have $B \subset \C_\infty \setminus \Omega \subset \C_\infty \setminus (K \cup f^N(L))$, since $L \subset D$ and $D$ is $f$-invariant, and $B \cap C \neq \varnothing$ for all components $C$ of $\C_\infty \setminus K \cup f^N(L)$.

	Next we consider the function
	\begin{equation*}
		\varphi(z) =
		\begin{cases}
			g(z), & \textrm{if } z \in K, \\
			\displaystyle \left( \prod_{k=1}^N (\omega \circ f^{-k}(z)) \right)^{-1} (h \circ f^{-N})(z), & \textrm{if } z \in f^N(L).
		\end{cases}
	\end{equation*}

	Since the disjoint sets $K$ and $f^N(L)$ are compact with $K \subset \Omega$ and $f^N(L) \subset f^N(U)$, and since $g \in H(\Omega)$, $h \in H(U)$, it follows that $\varphi$ can be extended holomorphically to an open neighbourhood of $K \cup f^N(L)$. Hence Runge's theorem yields a rational function $R$, with poles only in $B \subset \C_\infty \setminus \Omega$, such that
	\begin{equation*}
		\norm{\varphi - R}_{K \cup f^N(L)} < \frac{\varepsilon}{M+1}
	\end{equation*}
	where
	\begin{equation*}
		M \coloneqq \max_{z \in f^N(L)} \abs{\prod_{k=1}^N (C_{f^{-k}} \omega)(z)}.
	\end{equation*}

	The restriction $R|_\Omega$ is holomorphic on $\Omega$ and since $\abs{g(z) - R(z)} < \varepsilon$ for each $z \in K$, it follows that $R|_\Omega \in V_{\varepsilon,K,\Omega}(g) \subset \mathcal{V}$. 
	
	Furthermore, for each $z \in L$, we set $y \coloneqq f^N(z) \in f^N(L)$ and hence
	\begin{align*}
		\abs{h(z) - (\wcomp^N R)(z)} &= \abs{h(z) - \prod_{j=0}^{N-1} (\comp^j \omega)(z)\cdot  (\comp^N R)(z)} \\
		&= \abs{(h\circ f^{-N})(y) - \prod_{j=0}^{N-1} \left(\omega \circ f^{j-N} \right)(y) \cdot R(y)} \\
		&= \abs{\prod_{k=1}^N (\omega \circ f^{-k})(y) \cdot \varphi(y) - \prod_{k=1}^N \left(\omega \circ f^{-k}\right)(y) \cdot R(y)} \\
		&= \abs{\prod_{k=1}^N (C_{f^{-k}} \omega)(y)} \abs{\varphi(y) - R(y)} < \varepsilon.
	\end{align*}
So $\wcomp^N R \in V_{\varepsilon,L,U}(h) \subset \mathcal{W}$ and thus $\wcomp^N(\mathcal{V}) \cap \mathcal{W} \neq \varnothing$, which gives that the sequence $(\wcomp^n)_{n \in \N}$ is topologically transitive.
\end{proof}

With respect to a given $f$-evacuating subset $U$ of $D$, Theorem \ref{thm:wcompUniv} proves the $H(U)$-universality for $\wcomp$.  The following corollary expands on this to reveal that, under certain conditions, there exists an abundant supply of functions that induce $H(U)$-universality for all nonempty open subsets $U$ of $D$.

To state the next corollary, we adapt slightly the terminology of \cite[Definition 2.3]{Jun19}: for an open nonempty set $V \subset \C$, we say a sequence $(M_n)$ of subsets of $V$ is \emph{$V$-exhausting} if for each compact subset $K \subset V$, without holes, there exists $N \in \N$ with $K \subset M_N^\circ$. 
As discussed in \cite[p.~854]{Jun19} (cf.\ also \cite{BGMR95}), each open set $V \subset \C$ admits a $V$-exhausting sequence composed of open subsets, without holes, of $V$. 

The proof of Corollary \ref{cor:wcompAllSubsets} follows the argument of \cite[Corollary 2.4]{Jun19}; we include it here for the convenience of the reader and for completeness.

\begin{cor} \label{cor:wcompAllSubsets}
     Let $f \in H(D)$, $(V_m)_{m \in \N}$ be a $D$-exhausting sequence of subsets of $D$ that are evacuating for $f$, and for each $m \in \N$ let the restrictions of the iterates $f^n |_{V_m}$ be injective. Assume also that $\omega \in H(V_m)$ is non-vanishing for each $m\in \N$. Then the set of functions in $H(\Omega)$ that are $H(U)$-universal with respect to $\wcomp$ for all nonempty open subsets $U$ of $D$ without holes is comeagre in $H(\Omega)$.
\end{cor}

\begin{proof}
For $m\in \N$, let $\mathcal{G}_m$ denote the set of functions in $H(\Omega)$ that are $H(V_m)$-universal for $\wcomp$. By Theorem \ref{thm:wcompUniv}, each set $\mathcal{G}_m$ is comeagre in $H(\Omega)$, so it follows that the countable intersection $\mathcal{G} = \bigcap_{m\in \N} \mathcal{G}_m$ is also comeagre in $H(\Omega)$.

Next, we let $\mathcal{H}$ denote the set of functions in $H(\Omega)$ that are $H(U)$-universal for $\wcomp$ for all nonempty open subsets $U \subset D$ without holes. We aim to show the inclusion $\mathcal{G} \subset \mathcal{H}$.

Let $g \in \mathcal{G}$, $U$ be a nonempty open subset of $D$ without holes, $h \in H(U)$, $\varepsilon >0$ and $K$ be a compact subset of $U$.
Since $U$ does not contain holes, we have that the polynomially convex hull $\widehat{K}$ has no holes in $D$.
Runge's Theorem thus gives the existence of a polynomial $P$ with $\norm{P-h}_{\widehat{K}} < \varepsilon/2$.

Since $(V_m)$ is $D$-exhausting, there exists $M \in \N$ with $\widehat{K} \subset V_M$. Note that $P|_{V_M} \in H(V_M)$, and since $g \in \mathcal{G} \subset \mathcal{G}_M$, the $H(V_M)$-universality of $g$ means that there exists $N \in \N$ such that
\begin{equation*}
    \norm{\prod_{j=0}^{N-1} \left( \omega \circ f^j|_{V_M} \right) \cdot \left( g \circ f^N|_{V_M} \right) -P}_{\widehat{K}} < \varepsilon/2.
\end{equation*}
Hence
\begin{align*}
    \norm{\prod_{j=0}^{N-1} \left( \omega \circ f^j|_{V_M} \right) \cdot \left( g \circ f^N|_{V_M} \right) - h}_K &\leq  \norm{\prod_{j=0}^{N-1} \left( \omega \circ f^j|_{V_M} \right) \cdot \left( g \circ f^N|_{V_M} \right) - P}_{\widehat{K}} + \norm{P - h}_{\widehat{K}} \\
    &< \varepsilon/2 + \varepsilon/2 = \varepsilon
\end{align*}
and thus the function $g$ is $H(U)$-universal for $\wcomp$ and we obtain that $g \in \mathcal{H}$.
\end{proof}

 Note that the assumptions of Corollary \ref{cor:wcompAllSubsets} are satisfied if $f$ is injective on $D$ and $f^n|_D \to \partial_\infty\Omega$ locally uniformly, which gives the following corollary.

\begin{cor} \label{cor:f-injective-univ}
     Let $f \in H(D)$ be injective on $D$, $f^n|_D \to \partial_\infty\Omega$ locally uniformly and let $\omega \in H(D)$ be non-vanishing. Then the set of functions in $H(\Omega)$ that are $H(U)$-universal for $\wcomp$ for all nonempty open subsets $U \subset D$, with no holes, is a comeagre set in $H(\Omega)$.
\end{cor}

In the special case where $D$ does not contain holes, we also get the following corollary. 
\begin{cor} \label{cor:NoHoles-f-injective-univ}
     Let $D \subset \C$ be open with no holes. Assume that $f \in H(D)$ is injective  on $D$,  $f^n|_D \to \partial_\infty\Omega$ locally uniformly and let $\omega \in H(D)$. Then the set of functions in $H(\Omega)$ that are $H(D)$-universal for $\wcomp$ is a comeagre set in $H(\Omega)$.
\end{cor}

If we consider finite subsets of $D$, then the same approach as in Theorem \ref{thm:wcompUniv} allows us to deduce the following corollary.

\begin{cor} \label{cor:wcompFinite}
	Let $E \subset D$ be nonempty, finite and evacuating for $f$. Also assume that the restrictions $f^n |_E$ are injective for each $n \in \N$ and that $\omega \in C(E)$ is non-vanishing. Then the set of functions in $H(\Omega)$ that are $C(E)$-universal for $\wcomp$ is a comeagre set in $H(\Omega)$.
\end{cor}

\subsection{Applications}

In this subsection we illustrate our universality results for weighted composition operators $\wcomp$ through applications in the context of specific Fatou components of the symbol $f$.
We begin by considering the behaviour of $\wcomp$ with respect to attracting, superattracting and parabolic fixed points of $f$. 

The following result uses K\oe{}nigs's linearization theorem and we refer the reader to subsection \ref{subsec:holodyn} for a reminder of the details.

\begin{thm}
	Let $f$ be a transcendental entire function with an attracting fixed point at $z_0$. Let $U$ be an open neighbourhood of $z_0$ as per K\oe{}nigs's linearization theorem and assume that $\omega \in H(U \setminus\{ z_0 \})$ is non-vanishing.
	Then the set of functions in $H(\C \setminus\{ z_0 \})$ that are $H(W)$-universal for $\wcomp$ for all nonempty open $W \subset U\setminus\{ z_0 \}$, with no holes, is comeagre in $H(\C \setminus\{ z_0 \})$. 
\end{thm}

\begin{proof}
	K\oe{}nigs's linearization theorem gives that there exist open neighbourhoods $U$ of $z_0$, $V$ of 0 and a conformal map $\varphi \colon U \to V$ that conjugates $f|_U \colon U \to U$ with the map $F \colon V \to V$. 
	
	We consider the domain $\Omega = \C \setminus\{ z_0 \}$ and the open set $D = U \setminus\{ z_0 \} \subset \Omega$. 
	Since $F, \varphi$ and $\varphi^{-1}$ are injective, the same is true for $f|_U = \varphi^{-1} \circ F \circ \varphi$. The restriction $f|_D \colon D \to D$ is thus injective.
	Moreover, since $f^n |_D \to z_0 \in \partial_\infty \Omega$ locally uniformly, the assertion follows from Corollary \ref{cor:f-injective-univ}.
\end{proof}

Next we consider parabolic fixed points of $f$, cf.\ subsection \ref{subsec:holodyn} for a reminder of their pertinent properties.

\begin{thm} \label{thm:wcompPetals}
	Let $f$ be a transcendental entire function with a parabolic fixed point at $z_0$, let $P = \bigcup_{k=1}^m P_k$ be the union of its attracting petals and let $\omega \in H(P)$  be non-vanishing.
	Then the set of functions in $H(\C \setminus\{ z_0 \})$ that are $H(P)$-universal for $\wcomp$ is comeagre in $H(\C \setminus\{ z_0 \})$.
\end{thm}

Since $P$ has no holes, $f|_P \colon P \to P$ is injective and due to the uniform convergence $f^n|_P \to z_0 \in \partial_\infty \Omega$, Theorem \ref{thm:wcompPetals} thus follows by an application of Corollary \ref{cor:f-injective-univ}.

We next consider the case of a superattracting fixed point $z_0$ of the symbol $f$, where the open neighbourhood $U$ of $z_0$ is now according to B\"ottcher's theorem (cf.\ subsection \ref{subsec:holodyn}). However, the approach that we have hitherto employed to treat the cases of attracting and parabolic fixed points encounters an obstacle: the symbol $f$ is not injective on $U$. In order to have injectivity for each iterate $f^n$, we thus need to consider compact subsets $K \subset U \setminus \{z_0\}$ with empty interior (cf.\ the beginning of \cite[Section 3.3]{Jun19} for an account of the motivation for considering this setting).

    As per the notation introduced in the discussion surrounding B\"ottcher's theorem in subsection \ref{subsec:holodyn}, in the below corollary we consider the compact neighbourhood $B_r$ of $z_0$, for $r > 0$. The proof follows from an application of Corollary \ref{cor:wcompFinite} by taking $\Omega = \C \setminus \{ z_0 \}$ and  $D = B_r^\circ \setminus \{ z_0 \}$.
\begin{cor} \label{cor:WcompFinite01}
    Let $f$ be a transcendental entire function with a superattracting fixed point at $z_0$.
	Let $E \subset B_r^\circ \setminus \{ z_0 \}$ be a finite set such that each iterate $f^n$ is injective on $E$ and let $\omega \in C(E)$ be non-vanishing. Then the set of functions in $H(\C \setminus \{z_0\})$ that are $C(E)$-universal for $\wcomp$ is a comeagre set in $H(\C \setminus \{z_0\})$.
\end{cor}

By employing a similar approach to \cite[Theorem 3.5]{Jun19}, in the below theorem we are able to extend Corollary \ref{cor:WcompFinite01} to comeagre many compact subsets of $B_r^\circ \setminus \{ z_0 \}$. The neighbourhood $V$ that appears in the theorem statement corresponds to the open neighbourhood $V$ of the origin from B\"ottcher's theorem.

\begin{thm} \label{thm:superattracting1}
    Let $f$ be a transcendental entire function with a superattracting fixed point at $z_0$.
	For each $r>0$ with $\overline{\mathbb{D}}_r \subset V$, comeagre many functions in $H(\C \setminus \{z_0\})$ are $C(K)$-universal for $\wcomp$ for comeagre many compact subsets $K \subset B_r$, where $\omega \in C(K)$ is non-vanishing.
\end{thm}

We recall that comeagre many sets in the set of compact subsets of $B_r$ are Cantor sets, i.e.\ perfect and totally disconnected (cf.\ \cite[Remark 2]{BMM11}). So we have that $K^\circ = \varnothing$ and hence $A(K) = C(K)$ for comeagre many compact subsets $K \subset B_r$. Hence we may replace $C(K)$-universality by $A(K)$-universality in the statement of  
Theorem \ref{thm:superattracting1}.

We note in Theorem \ref{thm:superattracting1} that the comeagre subset of the set of compact subsets of $B_r$ depends on the choice of function from the comeagre subset of $H(\C \setminus \{z_0\})$. By applying \cite[Lemma 3.6]{Jun19}, this dependence can be interchanged which gives the following theorem.

\begin{thm} \label{thm:superattracting2}
    Let $f$ be a transcendental entire function with a superattracting fixed point at $z_0$.
	For each $r>0$ with $\overline{\mathbb{D}}_r \subset V$, comeagre many compact subsets $K \subset B_r$ are $C(K)$-universal for $\wcomp$ for comeagre many functions in $H(\C \setminus \{z_0\})$, where $\omega \in C(K)$ is non-vanishing.
\end{thm}

We next record results for weighted composition operators that are analogous to those from Section \ref{sec:BakerWandering} on invariant Baker domains, eventually isometric oscillating wandering domains and semi-contracting wandering domains of the symbol $f$.

Following the argument of Theorem \ref{thm:InvBaker} and applying Corollary \ref{cor:NoHoles-f-injective-univ} we get the following result for invariant Baker domains.
\begin{thm}
Let $f$ be a transcendental entire function, $U$ be an invariant Baker domain of $f$ and let $\omega \in H(U)$ be non-vanishing. Then there exists an unbounded domain $V \subset U$ such that the set of all entire functions which are $H(V)$-universal for $\wcomp$ is a comeagre set in $H(\mathbb{C})$.
\end{thm}
\begin{rmk} Note that in the case where $U$ is a univalent Baker domain we have $V=U$.\end{rmk}

Combining the arguments of Theorems \ref{thm:OscWanderingDoms} and \ref{thm:wcompUniv}, and \cite[Theorem 4.11(b)]{Jun19} we get the following result for eventually isometric wandering domains.
\begin{thm}
Let $f$ be a transcendental entire function, $(U_n)$ be a sequence of eventually isometric  escaping or oscillating wandering domains of $f$, and let $\omega \in H(\bigcup_n U_n)$ be non-vanishing.  Then the set of all entire functions that are $H(\bigcup_n U_n)$-universal for $\wcomp$ is a comeagre set in $H(\C)$.
\end{thm}

Arguing as in Theorem \ref{thm:semi} and applying Corollary \ref{cor:NoHoles-f-injective-univ}, we get the following theorem for semi-contracting wandering domains.
\begin{thm}
Let $f$ be a transcendental entire function and $(U_n)$ a sequence of semi-contracting wandering domains such that $f|_{U_n}$ has degree 2, for all $n \in \mathbb{N}$. Then there exists an open set $V_0 \subset U_0$ where all iterates of $f$ are injective. Hence, for $V= \bigcup_n f^n(V_0)$ and $\omega \in H(V)$ non-vanishing, the set of all entire functions which are $H(V)$-universal for $\wcomp$ is a comeagre set in $H(\C)$ \end{thm}

\begin{rmk}
It follows from \cite[Theorems 5.11, 6.1 and 7.1]{IMPAN}, \cite[Theorem 3.1]{Bes14} and the first part of Proposition \ref{mcwd}, that an analogous statement to the second part of Proposition \ref{mcwd} also holds for multiply connected wandering domains $(U_n)$ in the case of weighted composition operators. That is, there is no entire function which is universal in $\bigcup_{n=0}^{\infty} V_n$ for $\wcomp$, where $V_n \subset U_n$ and $f^n(V_0)=V_n$.
\end{rmk}

In the last part of this section we will focus on multiply connected domains.
Let $\Omega$ be a domain. A compact set $K$ will be called an \textit{isolated hole} of $\Omega$ if it is a bounded connected component of the complement of $\Omega$ and we can find a simple curve $\gamma$ in $\Omega$ such that $K$ is in the interior of $\gamma$. 
\begin{thm}
Let $\Omega$ be a domain that has at least one isolated hole and let $f$ be holomorphic on a neighbourhood of the hole. Then no cyclic (and so supercyclic) function exists for any weighted composition operator $W_{\omega, f}$. In particular, this is the case when $\Omega$ is finitely (but not simply) connected.
\end{thm}
\begin{proof}
Let $\omega$ be the weight; that is, a non-vanishing holomorphic function.
Also let $K$ be an isolated hole of $\Omega$ and let $a\in K$. Thus we can find a simple curve $\gamma$ in $\Omega$ such that $K$ is in the interior of $\gamma$. 

For the sake of contradiction, let us assume that there is a cyclic function $g$ in $\Omega$. Then, by the definition of cyclicity (cf.\ subsection \ref{subsec:LD}), there is a sequence $(g_n)$ of linear combinations of functions of the form $(g\circ f^k)\cdot\prod_{j=0}^{k-1} (\omega \circ f^j ) $ that converges uniformly on the trace $\Gamma$ of the curve $\gamma$ (which is a compact subset of $\Omega$) to $\dfrac{1}{z-a}$ (which is in $C(\Gamma)$). Since we have uniform convergence, we have convergence of the corresponding integrals; that is, 
\begin{align*}
\int _{\gamma} g_{n}(z) dz \to \int_{\gamma} \dfrac{1}{z-a} dz \ \ \ \mbox{ as } n\to\infty .
\end{align*} 
This leads to a contradiction since from Cauchy's integral formula we get that $\displaystyle\int_{\gamma} \dfrac{1}{z-a} dz =2\pi i$, while $\displaystyle\int _{\gamma} g_{n}(z) dz=0$ for each $n$.
\end{proof}

\section{Further remarks and applications} \label{sec:furtherRemarks}

We conclude with some remarks and further applications.

\subsection{Self-maps of the unit disc}
It is natural to consider questions on the universality of composition operators when the symbol is a holomorphic self-map of the unit disc $\mathbb{D}$.  Recall that in this case the behaviour of iterates of points in $\mathbb{D}$ is determined by the well-known Denjoy-Wolff Theorem.

\begin{thmA}[Denjoy-Wolff]
Let $f : \mathbb{D} \to \mathbb{D}$ be holomorphic. Provided $f$ is not a
rotation about a point in $\mathbb{D}$, there exists a unique point $p  \in \overline{\mathbb{D}}$ (the Denjoy–Wolff point) such that $f^{n}(z) \to p$  as $n \to \infty$ for all $z \in \mathbb{D}$.
\end{thmA}

\begin{prop}\label{prop:unit disc}
Let $f$ be a holomorphic self-map of $\mathbb{D}$ that is not a M\"obius transformation. Then there exists a domain $V \subset \mathbb{D}$ such that the set of holomorphic self-maps of $\mathbb{D}$ that are $H(V)$-universal for $C_f$ is a comeagre set in $H(\mathbb{D})$.
\end{prop}

\begin{proof}
Since $f$ is a holomorphic self-map of $\mathbb{D}$ it follows from \cite[Theorem 3.2]{Cowen} that there exists a domain $W \subset \mathbb{D}$ such that $f|_W$ is injective and $f(W) \subset W$, and so $f^n|_{W}$ is injective. 

In the case where the Denjoy-Wolff point, $p$, of $f$ lies in $\mathbb{D}$ we take $V$ to be any simply connected domain in $W$ and we deduce by \cite[Corollary 2.5]{Jun19} that the set of holomorphic self-maps of $\mathbb{D}$ that are $H(V)$-universal for $C_f$ is a comeagre set in $H(\mathbb{D})$. Similarly, if the Denjoy-Wolff point lies on $\partial \mathbb{D}$ we can use \cite[Corollary 2.6]{Jun19} to deduce that the same conclusion holds for $V=W$. \end{proof}

\begin{rmk}
 Note that in the case of M\"obius transformations $f$ is injective in $\mathbb{D}$. If $f$ is a rotation, then $f^n$ cannot be a run-away sequence, but in the case where $f$ has a Denjoy-Wolff point  on $\partial \mathbb{D}$ we can use \cite[Corollary 2.6]{Jun19} to obtain the same conclusion as in the case of non-M\"obius holomorphic self-maps of $\mathbb{D}$, with $V= \mathbb{D}$.
\end{rmk}

\subsection{Holomorphic functions with wild boundary behaviour}

 In \cite{Gonzalez}, the authors investigated the existence of large topological (and algebraic) structures inside families of holomorphic functions with full range. Moreover, in \cite{CS04}, Costakis and Sambarino studied the generic behaviour of entire functions as we approach infinity using the classes of universal functions in the sense of Birkhoff; i.e., entire functions that are universal with respect to the translation operators $T_a \colon H(\mathbb{C})\to H(\mathbb{C})$, defined by $(T_{a}f)(z)=f(z + a)$, where $a\in \mathbb{C}\setminus\{0\}$.  In the same spirit, we will derive analogous conclusions for the generic angular behaviour of holomoprhic functions on the punctured plane or on the unit disc using the existence of universal composition operators. In the next result, by triangular region with vertex at $\zeta$ we mean the intersection of an open disc centred at $\zeta$ with an open sector with vertex at $\zeta$.

\begin{ap}\label{Cor:disc}
\hangindent\leftmargini
\textup{(i)}
The set of holomorphic functions on the punctured plane $\mathbb{C}\setminus \{0\}$ which map any triangular region with vertex at $0$ to $\mathbb{C}$ is comeagre in $H(\mathbb{C}\setminus \{0\})$.
\begin{enumerate}[\normalfont (i)]
\setcounter{enumi}{1}
\item
The set of holomorphic functions on $\mathbb{D}$ which map any triangular region in $\mathbb{D}$ with vertex at any $\zeta\in \partial \mathbb{D}$ to $\mathbb{C}$ is comeagre in $H(\mathbb{D})$.
\end{enumerate}
\end{ap}

\begin{figure}[t]
\centering{\begin{tikzpicture}[scale=2]
\draw (0,0) circle [radius=1];
\draw[red] (0,0) to (0.34, 0.5);
\draw[red] (0,0) to (0.63,0.16);
\draw[dotted] (0,0) to (0.3,0.35);
\draw[dotted] (0.4,0.17) arc (-65:150:0.1);
\draw[dotted] (0,0) to (0.4,0.17);
\draw[red] (0.63,0.16) arc (10:70:0.45);
\draw[fill] (0.83,0.56) circle [radius=0.01];
\node at (0.97,0.62) {$q_m$};
\draw[fill] (0,0) circle [radius=0.01];
\node at (0,-0.2) {$0$};
\node at (0.5,-0.1) {$\color{red} S_{m,n,k}$};
\node at (-0.65, -0.5) {$\mathbb{D}$};
\end{tikzpicture}
}
\caption{The set up in the proof of Application  \ref{Cor:disc}(i)}\label{Figure1}
\end{figure}

\begin{proof}
\hangindent\leftmargini
\textup{(i)} 
Let $(q_m)$ be a dense sequence in $\partial\mathbb{D}$. For each $m$, $n$ and $k$, we consider the set $S_{m,n,k}$, defined as the intersection of the open disc $D(0,1/k)$ with the open sector with vertex at $0$, axis passing through $q_{m}$ and angle $2\pi /n$ (cf.\ Figure \ref{Figure1}). Let $\mathcal{U}_{m,n,k}$  be the set of all functions which are $H(S_{m,n,k})$-universal for $C_{f_{m,n,k}}$, where $f_{m,n,k}$ is an entire function which has a parabolic fixed point at $0$ such that one of its attracting petals at $0$ is contained in $S_{m,n,k}$. Since $\mathcal{U}_{m,n,k}$ is comeagre in $H(\mathbb{C}\setminus \{0\})$ for each choice of $m,n,k$ in $\mathbb{N}$, by Baire's category theorem, we derive that $\bigcap_{m,n,k}\mathcal{U}_{m,n,k}$, which clearly satisfies the required properties, is also comeagre in $H(\mathbb{C}\setminus \{0\})$. 

\begin{enumerate}[(i)]
\setcounter{enumi}{1}
\item The second part follows similarly, if we consider $S_{m,n,k}$ to be the intersection of the open disc $D(q_m,1/k)$ with the open sector with vertex at $q_m$, axis passing from $0$ and angle $2\pi /n$ and $f_{m,n,k}$ to be an entire function which has a parabolic fixed point at $q_m$ such that one of its attracting petals at $q_m$ is contained in $S_{m,n,k}$.
\end{enumerate}\end{proof}

\begin{rmk}

 We note that part (i) can be also derived using alternative approaches. For example, we could use the fact that the set of functions in $H(\mathbb{C}\setminus \{0\})$ which have an essential singularity at $0$ is comeagre together with some standard Baire category arguments. A result in this direction is Julia's Theorem \cite{julia}, which tells us that if $f$ is a holomorphic function on the punctured plane $\mathbb{C}\setminus \{0\}$ with essential singularity at $0$ then there exists a ray $R_{\theta}=\{z: \arg(z)=\theta\}$ such that $f$ assumes every finite value except possibly one on any sector of the form $\{z: |\arg(z)-\theta|<\varepsilon\}$.
\end{rmk}

\subsection{Universality of composition operators and stability}

 Given a sequence $(f_n)$ of holomorphic self-maps of the unit disc, we can consider the \textit{left-composition} sequence and the \textit{right-composition} sequence generated by $(f_n)$, defined by \begin{equation*} F_n=f_n\circ f_{n-1}\circ \dots \circ f_1 \ \text{ and } \ G_n =
f_1 \circ f_{2} \circ \cdots \circ f_{n}, \ n=1,2,\dots ,\end{equation*} respectively. Sequences of this type arise in a variety of contexts in complex dynamics and there is a growing literature about their asymptotic behaviour. We conclude this paper by giving an application of universal composition operators to the Denjoy-Wolff theory for such sequences of holomorphic self-maps of $\D$, whose boundary behaviour, unlike for sequences of iterates, can be wild. 

The stability of the Denjoy-Wolff point on $\partial \mathbb{D}$ when we consider left-composition and right-composition sequences has been recently studied in \cite{ChSh} and \cite{ArgyrisPhD}. In \cite{ChSh} the authors showed that if the convergence of a sequence $(f_n)$ to some function $f$
is sufficiently rapid, then the sequences $(F_n)$ of left-compositions of $(f_n)$ and the sequence $(f^n)$ of iterates of $f$  have similar dynamics. In particular, they showed that given any holomorphic self-map $f$ of $\D$ with a Denjoy–Wolff point $\zeta$ on the boundary of $\D$, there exists neighbourhoods $U_1, U_2, \dots$ of $f$ such that if $f_n\in U_n$ for $n=1,2,\dots$, then the left-composition sequence $F_n=f_n\circ f_{n-1}\circ \dots \circ f_1$ converges locally uniformly on $\D$ to $\zeta$. In the same article, the authors constructed an example which indicates that there is little hope of obtaining a simple analogue for right-composition sequences. In fact, the first author of that article in his thesis \cite[Theorem 3.10]{ArgyrisPhD} showed the following general result: given any holomorphic self-map $f$ of $\D$ with a Denjoy–Wolff point $\zeta$ on the boundary of $\D$, and given any $U_1, U_2, \dots$ neighbourhoods of $f$, there exists a sequence of holomorphic functions $(f_n)$ such that $f_k\in U_k$ for all $k=1,2,\dots$ and the right-composition sequence $G_n =
f_1 \circ f_{2} \circ \cdots \circ f_{n}$ diverges.

As we will see below, using universality of composition operators, we can draw a complementary conclusion when $f$ is additionally assumed to be injective.

\begin{ap}
Let $f$ be an automorphism of the unit disc $\mathbb{D}$ with a Denjoy-Wolff point $\zeta$ on the boundary of $\mathbb{D}$ and let $U_1, U_2, \dots$ be neighbourhoods of $f$ in $H(\D)$. Since $f$ is injective in $\mathbb{D}$ and $(f^n)$ is runaway, by Theorems \ref{thm:UnivCriterion} and \ref{thm:GEM09} we have that the set of $H(\mathbb{D})$-universal functions contains a dense $G_{\delta}$ set in $H(\mathbb{D})$. Hence there exists $g$ in $U_1$ which is $H(\mathbb{D})$-universal for $C_{f}$. Thus, if we put $f_1=g$, and $f_n=f$ for each $n>1$, the sequence of right-compositions $G_n =
f_1 \circ f_{2} \circ \cdots \circ f_{n}=g\circ f^{n-1}$ will be dense in $H(\mathbb{D})$, and so it will not converge to a point.  
\end{ap}

\section*{Acknowledgements}

The authors would like to thank Phil Rippon and Gwyneth Stallard for useful discussions, as well as St\'{e}phane Charpentier and Argyrios Christodoulou  for helpful comments.

%
%
%

\end{document}